\documentclass[10pt]{amsart}
\usepackage{amssymb}
\usepackage{amsfonts}
\usepackage{amscd}
\usepackage[dvips]{graphicx}
\usepackage[all,cmtip]{xy}
\usepackage{hyperref}
\usepackage{yhmath}
\usepackage{ulem}
\usepackage[all]{xy}
\usepackage{mathdots}
\usepackage{leftidx}
\usepackage{mathrsfs}
\usepackage{amsmath,amssymb, amsthm}
\usepackage{color}
\usepackage{tikz}
\usepackage{mathdots}
\usepackage{picture}
\usepackage{enumerate}
\usepackage{makecell}
\usepackage{extarrows}
\usepackage{graphicx}
\usepackage{caption}
\usepackage{subcaption}
\usepackage{float}
\numberwithin{equation}{section}
\newtheorem{Them}{Theorem}[section]
\newtheorem{Lem}[Them]{Lemma}
\newtheorem{Def}[Them]{Definition}
\newtheorem{Cor}[Them]{Corollary}
\newtheorem{Prop}[Them]{Proposition}
\newtheorem{Ex}[Them]{Example}
\newtheorem{Rem}[Them]{Remark}

\newcommand{\add}{{\mathsf{add}}}
\newcommand{\End}{{\mathsf{End}}}

\newcommand{\ind}{{\mathsf{ind}}}

\newcommand{\Id}{{\mathsf{Id}}}
\newcommand{\I}{{\mathsf{Im}}}

\newcommand{\soc}{{\mathsf{soc}}}
\renewcommand{\top}{{\mathsf{top}}}

\newcommand{\D}{{\mathsf{D}}}

\newcommand{\m}{\mathsf{mod}}
\newcommand{\stmod}{ \mathsf{\underline{mod}}}

\newcommand{\Hom}{{\mathsf {Hom}}}

\newcommand{\StHom}{\mathsf{\underline{Hom}}}
\newcommand{\Res}{{\mathsf {Res}}}
\newcommand{\Le}{{\mathsf {L_{e}}}}

\title[On reduction of simple-minded systems]{On reduction and gluing technique of simple-minded systems  over  self-injective algebras}
\author{Zhen Zhang}

\address{Zhen Zhang
	\newline Faculty of Arts and Science
	\newline Beijing Normal University 
	\newline  Zhuhai 519087
	\newline P.R.China}
\email{zhangzhen@bnu.edu.cn}

\date{version of \today}

\newcommand{\sdp}{\times\kern-.2em\vrule height1.1ex depth-.05ex}

 \normalbaselines
\thanks{The authors are supported by NSFC (No. 12031014 and No.12301044).}

\begin{document}
	\renewcommand{\thefootnote}{\alph{footnote}}
	\renewcommand{\thefootnote}{\alph{footnote}}
	\setcounter{footnote}{-1} 
	\renewcommand{\thefootnote}{\alph{footnote}}
	\setcounter{footnote}{-1}

\begin{abstract}
Let $A$	be a self-injective algebra over an algebraically closed field. We study  reduction of simple-minded systems over stable module category $A$-$\stmod$.  We present a recollement and study gluing technique of simple-minded systems through  a subset of a simple-minded system in $A$-$\stmod$. As a byproduct, we also study the extendible property of simple-minded systems.

\end{abstract}

\maketitle
\section{Introduction}
 Koenig and Liu \cite{KL} introduced simple-minded systems in the stable module category of any artin algebra. Later, Dugas \cite{Dugas} defined simple-minded systems over any triangulated category. The authors \cite{KL, Dugas, Z} studied properties of simple-minded systems, including the finite cardinality of a simple-minded system,  invariance of a simple-minded system under stable equivalences, mutation theory of simple-minded systems and relationship between simple-minded systems and weakly simple-minded systems, and so on.  There is a recent rise of interests in studying $w$-simple-minded systems (see, for example, \cite{CP,J, BCPW}), which was  introduced by Coelho Sim\~oes \cite{C} over a Hom-finite, Krull-Schmidt  $k$-linear triangulated category, where integer $w\geq1$.  $w$-simple-minded system is  a generalization of simple-minded system defined by Dugas.   Coelho Sim\~oes  and Pauksztello \cite{CP} study properties of  $w$-simple-minded systems, especially, they provided a reduction technique of  $w$-simple-minded systems and their construction stated  an induction technique for constructing $w$-simple-minded systems. Jin \cite{J} established the reduction of simple-minded collections (introduced by Koenig and Yang \cite{KY}) and he proved that the reduction of $w$-simple-minded systems  due to Coelho Sim\~oes  and Pauksztello is induced by the reduction of simple-minded collections.  In this paper, we concern about  reduction technique of  simple-minded systems over the stable module category of self-injective algebras.
This is the first  main result of our paper.
\begin{Them}{\rm(}Theorem \ref{new-triangulated-category-2}{\rm)}\label{new-triangulated-category-1}
Let  $A$ be a  self-injective $k$-algebra and  let $\mathcal{S}=\{S_1,\cdots\!,\, S_n\}$ be the set of simple modules on $A$-$\m$. Let $\mathcal{S'}=\{S_1,\cdots\!,\, S_m\}$ be a Nakayama-stable subset of  $\mathcal{S}$, that is, $\nu(\mathcal{S'})=\mathcal{S'}$, where $\nu$ is Nakayama functor.
Take $$\mathcal{D}={^{\bot}(\mathcal{S'})^{\bot}}=\{Z\in A{\text-}\stmod\, |\, \StHom_A(S,Z)=0, \    \StHom_A(Z,S)=0, S\in\mathcal{S'} \}.$$ Then  $\mathcal{D}$ is a triangulated category {\rm(}with a new structure{\rm)}, moreover, $\mathcal{D}$ and the stable module category  $eAe$-$\stmod$ of $eAe$ are equivalent as triangulated categories, where $e=e_{m+1}+\cdots+e_{n}$, $m<n$, $e_{i}$ is the primitive idempotent element corresponding to simple $A$-module $S_i$.	
\end{Them}

The cardinality of  simple-minded systems is closely related to Auslander-Reiten conjecture (AR-conjecture for short), which states that stable equivalences preserve the number of isomorphisms classes of non-projective simple modules. Precisely, if the cardinality of any simple-minded system over an finite dimensional algebra is a fixed number (the number of non-isomorphic and non-projective simple modules), then AR-conjecture holds for any finite dimensional algebra.  Reduction technique of  simple-minded systems provides possibility to solve AR-conjecture. 
 
 Recollements of triangulated categories, introduced by Beilinson, Bernstein and
 Deligne \cite{BBD} in their fundamental work on perverse sheaves. Since then, it has been a significant tool in algebraic geometry, topology and  in representation theory. For instance, Happel  \cite{H} established reduction techniques for homological conjectures by recollements of bounded derived module categories. More recently, Sun and Zhang \cite{SZ} provided a technique to glue simple-minded collections
along a recollement of a Hom-finite Krull-Schmidt triangulated
category. 
 Please see more research \cite{AKLY,CX, WWZ} on study of recollements. In this paper, we present a recollement through a particular subset of a simple-minded system over the stable module category of any self-injective algebra and we study the gluing technique of simple-minded systems through simple-minded systems of  outer terms in the recollement. 
 
 \begin{Them}{\rm(}Theorem \ref{recollement-sms}{\rm)}\label{recollement-sms-0}
 Let $A$ be a self-injective algebra and $\mathcal{S}$ a Nakayama-stable orthogonal system in $A$-$\stmod$. If  $\mathcal{F}(\mathcal{S})$ is functorially finite in $A$-$\stmod$ such that  $\Omega(\mathcal{S})$ is contained in $\mathcal{F}(\mathcal{S})$.  Then 
 \begin{enumerate}[$(1)$]
 \item  $A$-$\stmod$ is a recollement of $\mathcal{F}(\mathcal{S})$ and $\mathcal{S}^{\bot}$ as follows.
 \begin{align}
 \xymatrix{\mathcal{F}(\mathcal{S})\ar^-{i_*=i_!}[rr]&&A{\text-}\stmod\ar^-{j^!=j^*}[rr]\ar^-{i^!}@/^1.2pc/[ll]\ar_-{i^*}@/_1.6pc/[ll]&&\mathcal{S}^{\bot}\ar^-{j_*}@/^1.2pc/[ll]\ar_-{j_!}@/_1.6pc/[ll]}.
 \end{align}
 \item Let $\mathcal{M}=\{M_{i}\mid i=1,\cdots,m\}$ be a simple-minded system of $\mathcal{F}(\mathcal{S})$ and $\mathcal{N}=\{N_{j}\mid i=1,\cdots,n\}$ a simple-minded system of $\mathcal{S}^{\bot}$.  Then $\mathcal{S}'=\{i_{\ast}(M)\mid M\in\mathcal{M}\}\cup\{j_{\ast}(N)\mid N\in\mathcal{N}\}\}$ a simple-minded system in  $A$-$\stmod$.
 \end{enumerate}
\end{Them}	
 
 We also study the extendible property of  simple-minded systems as follows. 
 
 \begin{Them}{\rm(}Theorem \ref{sms-extending}{\rm)}\label{sms-extending-0}
 Let $A$ be a self-injective algebra and $\mathcal{M}$ a  Nakayama-stable  orthogonal system in $A$-$\stmod$ such that $\Omega(\mathcal{M})\subseteq\mathcal{F}(\mathcal{M})$ and $\mathcal{F}(\mathcal{M})$ is functorially finite in $A$-$\stmod$.  Then $\mathcal{M}$ extends to  a simple-minded system in $A$-$\stmod$ if and only if there is a simple-minded system in $^{\bot}\mathcal{M}^{\bot}$.
 \end{Them}

 This paper is organized as follows. In Section 2, we recall simple-minded systems, recollements, $w$-simple-minded systems, and related definitions and conclusion. In Section 3, we prove one of main results Theorem \ref{new-triangulated-category-2}. In Section 4, we mainly prove Theorem \ref{recollement-sms} and Theorem  \ref{sms-extending}.
  
\section{Preliminary}
Throughout this paper, we assume that all triangulated categories are  $k$-linear, Hom-finite, and Krull-Schmidt, and we assume that all algebras  are  finite dimensional  over algebraically closed field $k$. For an algebra $A$, we denote by $A$-$\m$ the category of finite dimensional left $A$-modules. The stable module category of $A$-$\m$ is denoted by $A$-$\stmod$, it has the same class of objects with $A$-$\m$, and for two objects $X,Y$ in $A$-$\stmod$, the abelian group   $\StHom_A(X,Y)$ from $X$ to $Y$ is the quotient  $\StHom_A(X,Y)/\mathcal{P}_{A}(X,Y)$, where $\mathcal{P}_{A}(X,Y)$ is the subgroup of $\StHom_A(X,Y)$ consisting of all $A$-module homomorphisms  which factor through a projective  $A$-module.  We  denote by $\add(M)$  all direct summands of finite direct sums of modules from $M$, and Note that, for a class of objects $\mathcal{M}=\{M_{i}\mid i\in I\}$ in an additive category $\mathcal{C}$, we simply denote $\add(\bigoplus_{i\in I}M_{i})$ by $\add(\mathcal{M})$.  In particular, $\mathcal{P}_{R}$ (resp. $\mathcal{I}_{R}$) is denoted by the full subcategories of finite generated  projective (resp. injective) $R$-modules $\add(_{R}R)$ (resp. $\add(\D(R_{R}))$). We denote by $\D$ the usual $k$-duality $\Hom_{k}(-,k)$. Let $A$ be a self-injective algebra. We denote by $\Omega$ (resp. $\Sigma$) the syzygy functor (resp. cosyzygy functor) which assigns to any object $M$ of $A$-$\stmod$ the kernel of its projective cover $P_{A}(M)\twoheadrightarrow M$ (cokernel of its injective enveloping $M\hookrightarrow I_{A}(M)$) in $A$-$\m$. It is known that stable module category $A$-$\stmod$ is a triangulated category with suspension functor $\Omega^{-1}$), and its distinguished triangles are induced by short exact sequences in $A$-$\m$. 

Let $A$ be a finite dimensional algebra and $e$ an idempotent element of $A$. Consider the algebra $B=eAe\cong \End(Ae)$ with identity element $e$. We recall two special functors as follows.
One is the {\bf restriction functor} \[\Res_{e}(-)=e(-):A{\text-}\m\rightarrow B{\text-}\m,\]
and the other is {\bf idempotent embedding functor}  \[\Le(-)=\Hom_{B}(eA,-):B{\text-}\m\rightarrow A{\text-}\m.\]

We summarize some properties of those two functors as follows.
\begin{Prop}{\rm(}\cite[Chapter I, Section 6]{ASS}{\rm)}\label{restriction-embedding}
Let $A$ be a finite dimensional algebra, $e$ an idempotent element of $A$ and let $B$ be the algebra $eAe$. We have the following results. 
\begin{enumerate}[$(1)$]
\item Idempotent embedding functor $\Le(-)$ is full and faithful such that $\Res_{e}\Le\cong1_{B{\text-}\m}$. 
\item  The functor $\Le(-)$ is right adjoint to the functor $\Res_{e}(-)$, that is, there is functorial isomorphism:
\[\Hom_{A}(_{A}X,\Le(_{B}Y))\cong\Hom_{B}(\Res_{e}(_{A}X),_{B}Y),\]
for any $A$-module $_{A}X$, and $B$-module $_{B}Y$. Note that  $\Res_{e}(-)$ is an exact functor  and  $\Le(-)$ a right exact functor.  
\item $\Le(-)$ preserves indecomposability and it carries injective objects to injective objects. Note that $\Res_{e}(-)$ does not preserve projective modules, however, for any object $P\in\add(Ae)$, $e(P)\in\add(B)$ is a projective $B$-module.  
\end{enumerate}
\end{Prop}

\subsection{Simple-minded systems and recollements}
Let $\mathcal{T}$ be a Hom-finite, Krull-Schmidt  $k$-linear triangulated category, $\Sigma$ the suspension functor of $\mathcal{T}$, denoted usually by $[1]$.  For any families $\mathcal{S}_{1}, \mathcal{S}_{2}$ of objects in $\mathcal{T}$, we define a family of objects
$$\mathcal{S}_{1}\star\mathcal{S}_{2}:=\{ X\in\mathcal{S}\mid \mbox{ There is a distinguished triangle }S_{1} \longrightarrow  X \longrightarrow S_{2} \longrightarrow S_{1}[1],  S_{1}\in \mathcal{S}_{1}, S_{2}\in \mathcal{S}_{2}\}. $$

Using the octahedral axiom,  the operator $\star$ is associative, that is,  $(\mathcal{S}_{1}\star\mathcal{S}_{2})\star\mathcal{S}_{3}=\mathcal{S}_{1}\star(\mathcal{S}_{2}\star\mathcal{S}_{3})$ for $\mathcal{S}_{1}, \mathcal{S}_{2}$ and $\mathcal{S}_{3}\subseteq\mathcal{T}$.  We denote $(\mathcal{S})_{0}=\{0\}$, and for $n\in\mathbb{Z}^{+}$, we inductively define $(\mathcal{S})_{n}=(\mathcal{S})_{n-1}\star(\mathcal{S}\cup\{0\})$. Dugas \cite[Lemma 2.3]{Dugas} showed that $(\mathcal{S})_{n}\star(\mathcal{S})_{m}=(\mathcal{S})_{n+m}$ for any non-negative integers $m$ and $n$. Similarly, one can define $ _{n}(\mathcal{S})$, and we have $(\mathcal{S})_{n}$=$_{n}(\mathcal{S})$.  We say that $\mathcal{S}$
is {\bf extension-closed}, if $\mathcal{S}\star\mathcal{S}\subseteq \mathcal{S}$. One denotes the extension closure of a family $\mathcal{S}$ of objects in $\mathcal{T}$ as $$\mathcal{F}(\mathcal{S}):=\bigcup_{n\geq0}(\mathcal{S})_{n},$$
which is the smallest extension closed full subcategory of $\mathcal{T}$ containing $\mathcal{S}$. 

For any family $\mathcal{S}$ of objects in $\mathcal{T}$, we set $$\mathcal{S^{\perp}}:=\{Y\in\mathcal{T}\mid\mathcal{T}(X,Y)=0, \forall X\in \mathcal{S}\},$$
$$\mathcal{^{\perp}S}:=\{Y\in\mathcal{T}\mid\mathcal{T}(Y,X)=0, \forall X\in \mathcal{S}\}.$$ We know that both $\mathcal{S^{\perp}}$ and $\mathcal{^{\perp}S}$ are extension closed subcategories of $\mathcal{T}$ as well as closed under direct summands. We shall denote $\mathcal{S^{\perp}}\cap\mathcal{^{\perp}S}$ by $\mathcal{^{\perp}S^{\perp}}$, which is called a {\it stable bi-perpendicular category}.

\begin{Def}\label{brick-orthogonal-system}
Let $\mathcal{T}$ be a Hom-finite, Krull-Schmidt,  $k$-linear triangulated category.  An object $M$ in $\mathcal{T}$ is a stable brick if $\mathcal{T}(M,M)\cong k$.   Moreover, a family $\mathcal{S}$ of stable bricks in $\mathcal{T}$ is an {\bf orthogonal system} if $\mathcal{T}(M,N)=0$ for all distinct  $M, N$ in $\mathcal{S}$.
\end{Def}

\begin{Def}{\rm(\cite[Definition 2.4, 2.5]{Dugas}, \cite[Definition 2.1]{KL})} \label{definition-sms-representation-finite} Let $\mathcal{T}$ be a Hom-finite, Krull-Schmidt  $k$-linear triangulated category. A family of objects $\mathcal{S}$ in $\mathcal{T}$ is a {\bf simple-minded system},  if the following two conditions are satisfied$\colon$
\begin{enumerate}[$(1)$]
\item {\rm(Orthogonality)} $\mathcal{S}$ is an orthogonal system in $\mathcal{T}$. 
\item {\rm(Generating condition)} $\mathcal{F}(\mathcal{S})=\mathcal{T}$.
\end{enumerate}
\end{Def}

\begin{Def} {\rm(\cite[Definition 3.1]{Dugas})} \label{torsion-pair}
Let $\mathcal{T}$ be a  triangulated category with suspension functor $[1]$ and let $\mathcal{X}, \mathcal{Y}$ be two additive subcategories of $\mathcal{T}$ which are closed under direct summands and isomorphisms. The pair $(\mathcal{X},\mathcal{Y})$ is called a {\bf torsion pair}, if 
\begin{enumerate}[$(1)$]
\item $\mathcal{T}(\mathcal{X},\mathcal{Y})=0$.
\item For any $C\in\mathcal{T}$, there  is  a  triangle		
\[\xymatrix{X_{C}\ar[r]^-{} & C\ar[r]^-{} &  Y^{C}\ar[r]^-{}& X_{C}[1]}\]
with $X_{C}\in\mathcal{X}$ and $Y^{C}\in\mathcal{Y}$.
\end{enumerate}
\end{Def}

We say a torsion  pair $(\mathcal{X},\mathcal{Y})$ 	in $\mathcal{T}$ is {\bf splitting}  provided that, for any indecomposable object $M$ in $\mathcal{T}$, $M$ is either in 
$\mathcal{X}$ or in $\mathcal{Y}$.

\begin{Def}{\rm(\cite{BBD})}\label{t-structure} 
Let $\mathcal{T}$ be a  triangulated category with suspension functor $[1]$. A pair  $(\mathcal{T}^{\leq0}, \mathcal{T}^{\geq0})$ of  subcategories of $\mathcal{T}$ is called a {\bf t-structure}, if 
\begin{enumerate}[$(1)$]
\item $\mathcal{T}(X,Y)=0$ for $\forall X\in\mathcal{T}^{\leq0},Y\in\mathcal{T}^{\geq0}	[-1]$.
\item $\mathcal{T}^{\leq0}$ is closed under suspension functor $[1]$ and  $\mathcal{T}^{\geq0}$ is closed under cosuspension functor $[-1]$, that is, $\mathcal{T}^{\leq0}[1]\subseteq\mathcal{T}^{\leq0}$ and $\mathcal{T}^{\geq0}[-1]\subseteq\mathcal{T}^{\geq0}$.
\item For any $C\in\mathcal{T}$, there is  a   triangle		
\[\xymatrix{X\ar[r]^-{} & C\ar[r]^-{} &  Y\ar[r]^-{}& X[1]}\]
with $X\in\mathcal{T}^{\leq0}$ and $Y\in\mathcal{T}^{\geq0}[-1]$.
\end{enumerate}
\end{Def}

A t-structure $(\mathcal{U},\mathcal{V})$ of $\mathcal{T}$ is called {\bf stable} provided that both $\mathcal{U}$ and $\mathcal{V}$ are triangulated subcategories of  $\mathcal{T}$. It is known  that the triangle $(3)$ in Definition \ref{t-structure} is unique under triangulated equivalence. 

\begin{Rem}
\begin{enumerate}[$(1)$]
\item  In Definition \ref{torsion-pair}, we do not require the subcategories $\mathcal{X}$ and $\mathcal{Y}$ in torsion pair $(\mathcal{X},\mathcal{Y})$  satisfies condition $\mathcal{X}[1]\subseteq\mathcal{X}$, $\mathcal{Y}[-1]\subseteq\mathcal{Y}$.
\item For a t-structure 	$(\mathcal{T}^{\leq0}, \mathcal{T}^{\geq0})$, 	it is clear that $(\mathcal{T}^{\leq0}, \mathcal{T}^{\geq0}[-1])$ is a torsion pair.
\end{enumerate}
\end{Rem}

\begin{Def}\label{recollement} {\rm(\cite{BBD})} 
Let $\mathcal{T}$ be a  triangulated category and let $\mathcal{X}, \mathcal{Y}$ be two triangulated subcategories of $\mathcal{T}$. We say that  $\mathcal{T}$  is a {\bf recollement} of $\mathcal{X}$ and $\mathcal{Y}$ if there are six functors in the following diagram as follows.
\begin{align}\label{recoll}
\xymatrix{\mathcal{X}\ar^-{i_*=i_!}[rr]&&\mathcal{T}\ar^-{j^!=j^*}[rr]
	\ar^-{i^!}@/^1.2pc/[ll]\ar_-{i^*}@/_1.6pc/[ll]
&&\mathcal{Y}\ar^-{j_*}@/^1.2pc/[ll]\ar_-{j_!}@/_1.6pc/[ll]},
\end{align}
such that
\begin{enumerate}[$(1)$]
\item $(i^*,i_*),(i_!,i^!),(j_!,j^!)$ and $(j^*,j_*)$ are adjoint pairs;
\item $i_*,j_*$ and $j_!$ are fully faithful functors;
\item $i^!j_*=0$; 
\item for each object $T\in\mathcal{T}$, there are two triangles in $\mathcal{T}$
\[\xymatrix{i_!i^!(T)\ar[r]^{\ \ \ \varepsilon_{T}} &T\ar[r]^{\eta_{T} \ \ \ }& j_*j^*(T)\ar[r]^{}& i_!i^!(T)[1],}\]
\[\xymatrix{ j_!j^!(T)\ar[r]^{\ \ \ \omega_{T}} &T\ar[r]^{\zeta_{T}\ \ \ } & i_*i^*(T)\ar[r]^{} & j_!j^!(T)[1]}.\]
Where $\varepsilon_{T}$ is the counit of torsion pair $(i_!,i^!)$, $\eta_{T}$ is the unit of torsion pair  $(j^*,j_*)$, $\omega_{T}$ is the counit of torsion pair  $(j_!,j^!)$, $\zeta_{T}$ is the unit of torsion pair  $(i^*,i_*)$.
\end{enumerate}
\end{Def}

\begin{Them}$($\cite{Mi},  \cite[Proposition 11.7.1]{ZP}$)$\label{TTF-triple-recollement} 
Let $\mathcal{T}$ be a triangulated category.
\begin{enumerate}[$(1)$]	
\item Suppose that  $\mathcal{T}$ is a recollement  of $\mathcal{X}$ and  $\mathcal{Y}$ {\rm(}see diagram \ref{recoll}{\rm)}.  Take \[\mathcal{U}=\I(j_{!}), \mathcal{V}=\I(i_{\ast}),\ \text{and}\  \mathcal{W}=\I(j_{\ast}).\]
Then $(\mathcal{U},\mathcal{V})$ and $(\mathcal{V},\mathcal{W})$ are stable t-structures in $\mathcal{T}$.

\item Conversely, we assume there are two stable t-structures  $(\mathcal{U},\mathcal{V})$ and $(\mathcal{V},\mathcal{W})$ in $\mathcal{T}$. Then $\mathcal{T}$ is a recollement of $\mathcal{V}$ and $\mathcal{W}$ as follows. 
\begin{align}\label{recoll-1}
\xymatrix{\mathcal{V}\ar^-{i_*=i_!}[rr]&&\mathcal{T}\ar^-{j^!=j^*}[rr]\ar^-{i^!}@/^1.2pc/[ll]\ar_-{i^*}@/_1.6pc/[ll]&&\mathcal{W}\ar^-{j_*}@/^1.2pc/[ll]\ar_-{j_!}@/_1.6pc/[ll]}.
\end{align}
Where 
\begin{enumerate}[$(a)$]	
\item $i_{\ast}$ and $j_{\ast}$ are  embedding functors;
\item $i^{\ast}$, $i^{!}$ and $j^{\ast}$  are the triangulated functors induced by t-decomposition; \item $j_{!}$ is the composition 
\[\xymatrix{\mathcal{W}\ar[r]^{F} &\mathcal{U}\ar[r]^{\varepsilon}&\mathcal{T},}\]
where $\varepsilon$ is an  embedding functor, and $F$ is the composition of triangulated equivalence $\mathcal{W}\rightarrow \mathcal{T}/\mathcal{V}$ and triangulated equivalence $ \mathcal{T}/\mathcal{V}\rightarrow\mathcal{U}$ induced by $G\colon\mathcal{T}\rightarrow \mathcal{U}$, satisfying that $F^{-1}G=j^{\ast}$, $\mathcal{U}=\I(j_{!})$, $\mathcal{V}=\I(i_{\ast})$ and $\mathcal{W}=\I(j_{\ast})$. 
\end{enumerate}
\end{enumerate}
\end{Them}

Let $\mathcal{T}$ be an additive category and $\mathcal{X}$ a full subcategory of $\mathcal{T}$.  We say that a morphism $f\colon M\rightarrow X$ is a {\bf left $\mathcal{X}$-approximation} of $M$, if $X\in\mathcal{X}$, and the abelian group homomorphism  \[\mathcal{T}(f,X')\colon \mathcal{T}(X,X')\rightarrow \mathcal{T}(M,X')\]
 is surjective for each $X'\in\mathcal{X}$. Dually, a morphism $g\colon Y\rightarrow N$ is a {\bf right $\mathcal{X}$-approximation} of $N$, if $Y\in\mathcal{X}$, and the abelian group homomorphism \[\mathcal{T}(Y',g)\colon \mathcal{T}(Y',Y)\rightarrow \mathcal{T}(Y',N)\] 
 is surjective for each $Y'\in\mathcal{X}$.
 Let $\mathcal{S}$ be a full subcategory of $\mathcal{T}$ containing $\mathcal{X}$. $\mathcal{S}$ is said to be {\bf covariantly finite} (resp. {\bf contravariantly finite}) in $\mathcal{X}$ if every $S\in\mathcal{S}$ has a left (resp. right) $\add(\mathcal{X})$-approximation. We say $\mathcal{S}$  is  {\bf functorially finite} in  $\mathcal{X}$ provided that  $\mathcal{S}$ is both covariantly finite and contravariantly finite in $\mathcal{X}$.  The following theorems provide us some methods to construct torsion pairs.

\begin{Them}$($\cite[Proposition 2.3]{IY}$)$\label{two-torsion-pairs}
Let $\mathcal{T}$ be a Hom-finite Krull-Schmidt  triangulated category and $\mathcal{S}$ an orthogonal system in $\mathcal{T}$. If the subcategory $\mathcal{F}(\mathcal{S})$ is contravariantly finite {\rm(}resp. covariantly finite{\rm)} in $\mathcal{T}$, then  $(\mathcal{F}(\mathcal{S}),\mathcal{S^{\perp}})$ {\rm(}resp. $(\mathcal{^\perp S}, \mathcal{F}(\mathcal{S}))${\rm)} is a torsion pair in $\mathcal{T}$. In particular, if the subcategory $\mathcal{F}(\mathcal{S})$ is functorically  finite in $\mathcal{T}$, then both  $(\mathcal{F}(\mathcal{S}),\mathcal{S^{\perp}})$ and  $(\mathcal{^\perp S}, \mathcal{F}(\mathcal{S}))$  are torsion pairs in $\mathcal{T}$. 
\end{Them}

\begin{Them}$($\cite[Theorem 3.3]{Dugas}$)$\label{subset-of-sms}
Let $\mathcal{T}$ be a Hom-finite Krull-Schmidt  triangulated category. Suppose $\mathcal{X}\subseteq\mathcal{S}$ for a simple-minded system $\mathcal{S}$ in $\mathcal{T}$. Then $(^{\perp}\mathcal{X},\mathcal{F}(\mathcal{X}))$ and $(\mathcal{F}(\mathcal{X}),\mathcal{X}^{\perp})$ are torsion pairs in $\mathcal{T}$.  In particular, $ \mathcal{F}(\mathcal{X})$ is a functorially finite subcategory of $\mathcal{T}$.
\end{Them}	
A morphism $f:M\rightarrow N$ is said to be {\bf right minimal} if whenever $g:M\rightarrow M$ has the property that $fg=f$, then $g$ is an isomorphism.
Recall that, for a torsion pair $(\mathcal{X},\mathcal{Y})$, a triangle of the form $\xymatrix{X\ar[r]^{\alpha} &C\ar[r]^{\beta} &Y\ar[r]^{}& X[1]}$ is called a {\bf minimal $(\mathcal{X},\mathcal{Y})$-triangle} for $X$ if  $X\in\mathcal{X}$, $Y\in\mathcal{Y}$ and $\alpha$ is a minimal right $\mathcal{X}$-approximation. 
For a subset $\mathcal{X}$ of a simple-minded system $\mathcal{S}$,
by the above Theorem \ref{subset-of-sms}, we have two torsion pairs $(^{\bot}\mathcal{X},\mathcal{F}(\mathcal{X}))$ and $(\mathcal{F}(\mathcal{X}),\mathcal{X}^{\bot})$. 
With the notation as in \cite[Section 4]{Dugas}, we define ${\bf a}\colon\mathcal{T}\rightarrow{^{\bot}\mathcal{X}}$, ${\bf b, c}\colon\mathcal{T}\rightarrow\mathcal{F}(\mathcal{X})$, and ${\bf d}\colon\mathcal{T}\rightarrow\mathcal{X}^{\bot}$ via 
the minimal triangles
\[\xymatrix{{\bf a}X\ar[r]^{ } &X\ar[r]^{} &{\bf b}X\ar[r]^{}& {\bf a}X[1],}\xymatrix{{\bf c}X\ar[r]^{ } &X\ar[r]^{} &{\bf d}X\ar[r]^{}& {\bf c}X[1].}\]
by the above two torsion pairs.   Note that, if $\mathcal{F}(\mathcal{X})$ is functorially finite in $\mathcal{T}$, then the  above {\bf a}, {\bf b}, {\bf c} and {\bf d} are well-defined.

Let $\mathcal{T}$ be a  triangulated category.  A  self-equivalence functor  $\mathbb{S}$  over $\mathcal{T}$ is called {\bf Serre duality} provided that it satisfies the isomorphism
$\mathcal{T}(X,Y)\cong\D\mathcal{T}(Y,\mathbb{S}(X)),$ which is natural in $X$ and $Y$, and  $\D=\Hom_{k}(-,k)$ is  the $k$-duality functor. We always assume that all  triangulated categories admit Serre duality in this paper.

\begin{Lem}$($\cite[Lemma 4.6 and Lemma  4.7]{Dugas}$)$\label{subset-of-sms-bipen}
Let $\mathcal{T}$ be a Hom-finite Krull-Schmidt  triangulated category with Serre duality $\mathbb{S}$. Suppose $\mathcal{X}\subseteq\mathcal{S}$ for a simple-minded system $\mathcal{S}$ in $\mathcal{T}$ such that $\mathbb{S}(\mathcal{X}[1])=\mathcal{X}$. Then
\begin{enumerate}[$(1)$]
\item For any triangle ${\bf a}M \to M \to {\bf b}M \to {\bf a}M[1],$ if $M\in{\mathcal{X}^{\bot}}$, then ${\bf a}M\in{^{\bot}\mathcal{X}^{\bot}}$.
\item For any triangle  ${\bf c}M\to M\to {\bf d}M\to {\bf c}M[1]$, if $M\in{^{\bot}\mathcal{X}}$, then ${\bf d}M\in{^{\bot}\mathcal{X}^{\bot}}$.
\end{enumerate}  
\end{Lem}	
\begin{Rem}\label{bi-per-contra}
By the proof of Lemma 4.3 and Lemma 4.6 in \cite{Dugas}, it is routine to check that the above Lemma \ref{subset-of-sms-bipen} $(1)$ can be generalized as follows. If $\mathcal{X}$ is a Nakayama-stable orthogonal system and $\mathcal{F}(\mathcal{X})$  is contravariantly finite in $A$-$\stmod$. For any triangle ${\bf a}M \to M \to {\bf b}M \to {\bf a}M[1],$ if $M\in{\mathcal{X}^{\bot}}$, then ${\bf a}M\in{^{\bot}\mathcal{X}^{\bot}}$. There is also a dual result for Lemma \ref{subset-of-sms-bipen} $(2)$, we do not state here.
\end{Rem}
We present \cite{Z} a practical method to verify when an orthogonal system is a simple-minded system in the stable module category of a Brauer graph algebra. 
\begin{Them}$($\cite[Theorem  4.16]{Z}$)$\label{BGA-sms}
Let $A$ be a domestic Brauer graph algebra and  $\mathcal{S}$ an  orthogonal system in $A$-$\stmod$.  Then $\mathcal{S}$ is a simple-minded system if and only if  $\Omega(\mathcal{S})$  in  $\mathcal{F}(\mathcal{S})$ and  $\mathcal{S}$ contains at least one object for each Euclidean component. 	
\end{Them}

\subsection{$w$-simple-minded systems}
Coelho Sim\~oes \cite{C} introduced  $w$-simple-minded systems over a Hom-finite, Krull-Schmidt  $k$-linear triangulated category with Serre duality, where integer $w\geq1$.  
Now we recall the definition of a $w$-simple-minded system.


Let $\mathcal{T}$ be a triangulated category with the shift functor $\Sigma$ and Serre duality $\mathbb{S}$. Let $d\in\mathbb{Z}$.  We say $\mathcal{T}$ is  {\bf\,$d$-Calabi-Yau} ($d$-CY for short) provided that there is a natural  isomorphism  $\mathbb{S}\cong \Sigma^{d}$. We write $\mathbb{S}_{w}= \mathbb{S}\Sigma^{-w}$. A subcategory $\mathcal{X}$ of $\mathcal{T}$ is said to be a {\bf $\mathbb{S}_{w}$-subcategory} of $\mathcal{T}$ if $\mathcal{X}=\mathbb{S}_{w}(\mathcal{X})=\mathbb{S}_{w}^{-1}(\mathcal{X})$.

\begin{Def}{\rm(\cite[Section 2]{CP})}\label{w-sms} Let $\mathcal{T}$ be a  Hom-finite, Krull-Schmidt,  $k$-linear triangulated category.
Let $\mathcal{S}=\{S_i\,|\,i\in I\}$ be the set of indecomposable objects of $\mathcal{T}$ and $w$ a positive  integer. 
\begin{enumerate}[$(1)$]
\item A collection of objects  $\mathcal{S}$ of $\mathcal{T}$ is  {\bf$w$-orthogonal}, if 
\begin{enumerate}[$(i)$]
\item For any  $S_i,\,S_j\in\mathcal{S}$, $\mathcal{T}(S_i,S_j)\cong\left\{\begin{array}{ll}  k &\mbox{if }\ i=j;\\
0 & \mbox{otherwise}.\end{array}\right. $
\item If $w\geq 2$, then we have  $\mathcal{T}(\Sigma^kS_i,S_j)=0$ for $1\leqslant k\leqslant w-1$, $S_i,\,S_j\in \mathcal{S}$.
\end{enumerate}
	
	
\item A $w$-orthogonal collection $\mathcal{S}$ of $\mathcal{T}$ is called a {\bf\,$w$-simple-minded system}, if 
\[\mathcal{T}=\add(\mathcal{F}(\mathcal{S}\cup\Sigma^{-1}\mathcal{S}\,\cup,\cdot\cdot\cdot,\cup\,\Sigma^{1-w}\mathcal{S})).\]
\end{enumerate}
\end{Def}

\begin{Rem}\label{CY-sms}
\begin{enumerate}[$(1)$]
\item If  $\mathcal{T}$  is the stable module {\rm(}projective{\rm)} category of a self-injective algebra,  then $1$-simple-minded systems defined by Coelho Sim\~oes  coincide with simple-minded systems introduced by Koenig-Liu or Dugas {\rm(}please refer to \cite{KL} or \cite{Dugas}{\rm)}. 

\item If $\mathcal{T}$ is the stable module category $A$-$\stmod$  of a self-injective  algebra $A$, then a $\mathbb{S}_{-1}$-subcategory $\mathcal{X}$ of $\mathcal{T}$ satisfies condition $\mathbb{S}_{-1}(\mathcal{X})=\mathbb{S}\Sigma(\mathcal{X})=\nu(\mathcal{X})=\mathcal{X}$, that is, $\mathcal{X}$  is 
Nakayama-stable, where $\nu$ is Nakayama functor $\D(_{A}A)\otimes_{A}-$ and  $\mathbb{S}=\nu\Omega=\nu\Sigma^{-1}$ is Serre duality.
Especially, if  $\mathcal{T}$ is the stable module category $A$-$\stmod$  of a symmetric algebra $A$,  then $A$-$\stmod$  is $(-1)$-CY and  any subcategories of  $\mathcal{T}$ is a $\mathbb{S}_{-1}$-subcategory.
\end{enumerate}
\end{Rem}


Inspired by reduction of cluster-tilting objects studied by Iyama and Yoshino \cite{IY}, Coelho Sim\~oes and  Pauksztello \cite{CP} studied reduction of $w$-simple-minded systems and reduction of positive Calabi-Yau triangulated categories.   We recall some definitions before stating their results.

\begin{Them}{\rm(\cite[Theorem  A]{CP})}\label{new-triangulated-category-1-1}
	Let  $\mathcal{T}$ be a Hom-finite, Krull-Schmidt, $k$-linear triangulated category and $w\geq1$. Suppose  $\mathcal{S}$ is a  $w$-orthogonal collection whose extension closure  $\mathcal{F(S)}$ is functorially finite. Moreover assume  $\mathcal{S}$ is a $\mathbb{S}_{-w}$-subcategory of $\mathcal{T}$. Let
	\begin{align*}
		\mathcal{D}=\{D\in \mathcal{T}\,|\, \mathcal{T}(\Sigma^{-i}S,D)=0, \  \mathcal{T}(D,\Sigma^{-i}S)=0,\  for\   \  i=0,\cdot\cdot\cdot,w-1, S\in\mathcal{S} \}.
	\end{align*}
	Then  $\mathcal{D}$ is a triangulated category.
\end{Them}
Note that  $\mathcal{D}$ is not a triangulated subcategory of $\mathcal{T}$. Coelho Sim\~oes and Pauksztello define a new triangulated structure on  $\mathcal{D}$ (cf. \cite[Section 4]{CP}). When $\mathcal{T}$  is a stable module  category of a self-injective algebra, we state triangulated structure of  $\mathcal{D}$ for $w=1$ as follows.

First, if  $\mathcal{T}$ is the stable module category $A$-$\stmod$ of  a self-injective algebra $A$ and  $w=1$, then  $\mathbb{S}_{-1}$ is  Nakayama functor,  that is,  $\mathcal{S}$ satisfies  Nakayama-stable condition.   In this case, \[\mathcal{D}=\{X\in A\text{-}\stmod\,|\, \StHom_{A}(\mathcal{S},X)=0, \  \StHom_{A}(X,\mathcal{S})=0\},\]
that is,  $\mathcal{D}$ is the stable bi-perpendicular category of  $\mathcal{S}$.

Then we introduce the shift functor and distinguished triangles of $\mathcal{D}$. The action of shift functor $[1]'$ of  $\mathcal{D}$ is as follows. Take $X\in\mathcal{D}$. Consider the triangle in $A$-$\stmod$: \[\xymatrix{S_{X} \ar[r]^-{\alpha_{X}} &  X[1] \ar[r]^-{\beta_{X}} & Y\ar[r]^-{\gamma_{X}} & S_{x}[1],}\]
where $\alpha_{X}$ is a right minimal  $\mathcal{F(S)}$-approximation of $X[1]$.  We define  $X[1]':=Y$.  
$[1]'$ is defined on morphism as follows. Given a morphism $f: X\rightarrow Y$ on $\mathcal{D}$. Consider the following commutative diagram:

\[\xymatrix{S_{X}\ar@{-->}[d]_{\sigma} \ar[r]^{\alpha_{X}} & X[1]\ar[d]_{f[1]} \ar[r]^{\beta_{X}} &X[1]'	 \ar[d]_{g} \ar[r]^{\gamma_{X}} & S_{X}[1] \ar[d]^{\sigma[1]} \\
	S_{Y} \ar[r]^{\alpha_{Y}} & Y[1]\ar[r]^{\beta_{Y}} &Y[1]'	 \ar[r]^{\gamma_{Y}} &  S_{Y}[1].}\]
We define $[1]'f:=g$. 

The distinguished triangle on $\mathcal{D}$ is as follows:
Given a morphism $f\colon X\rightarrow Y$ on $\mathcal{D}$. Extend  $f$ to be a triangle on $A$-$\stmod$ and consider the following commutative diagram:
$$\xymatrix{
	&   & s_f \ar[d]_-{\alpha} \ar[r]^-{\sigma} & s_X \ar[d]^-{\alpha_X}  \\
	X \ar[r]^-{f} & Y \ar[dr]_-{\beta g_1} \ar[r]^-{g_1} &Z_f \ar[d]_-{\beta} \ar[r]^-{h_1} &  X[1] \ar[d]^-{\beta_X} & (\Diamond)
	\\
	&   & Z \ar[d]_-{\gamma} \ar[r]^-{h} & X[1]' \ar[d]^-{\gamma_X}  \\	
	& &  s_f[1] \ar[r]_-{\sigma[1]} &  s_X[1],\\
	& & (\vartriangle) &  (\Box)}$$
where $\alpha$ (resp. $\alpha_X$) is a right minimal   $\mathcal{F(S)}$-approximation of $Z_f$ (resp. $ X[1]$);  $(\Diamond)$, $(\Box)$ and  $(\vartriangle)$ are triangles of  $A$-$\stmod$. We define  
\[\xymatrix{X \ar[r]^-{f} &  Y \ar[r]^-{\beta g_{1}} & Z\ar[r]^-{h} & X[1]'}\]
to be {\bf standard distinguished triangle} on $\mathcal{D}$. Let $\varepsilon'$ be the set of sequences of the form \[\theta\colon\!\!\xymatrix{M \ar[r]^-{a} &  N \ar[r]^-{b} & H\ar[r]^-{c} & M[1]',}\]  where $M,N$ and $H$ are in $\mathcal{D}$, and  $\theta$ is isomorphic to a standard distinguished triangle. {\bf $(\mathcal{D},[1]',\varepsilon')$ is a triangulated category.}

The following theorem provides a reduction of $w$-simple-minded systems.
\begin{Them}{\rm(\cite[Theorem B]{CP})}\label{sms-1-1-corresponding-1} 
Let  $\mathcal{T}$ be a Hom-finite, Krull-Schmidt, $k$-linear triangulated category and $w\geq1$. Suppose  $\mathcal{S}$ is a  $w$-orthogonal collection whose extension closure  $\mathcal{F(S)}$ is functorially finite and  $\mathcal{D}$ be as in Theorem \ref{new-triangulated-category-1-1}. Then there is a bijection, 
	\begin{equation*}
		\{w{\text-}simple{\text-}minded\  systems\  in\  \mathcal{T}\ containing\  \mathcal{S}\}\longleftrightarrow\{w{\text-}simple{\text-}minded\  systems\  in\ \mathcal{D}\}
	\end{equation*}
Moreover $\sigma$ takes $\mathcal{M}$ to $\mathcal{M}\backslash\mathcal{S}$,  where $\mathcal{M}$ is a  $w$-simple-minded system containing $\mathcal{S}$. Conversely, $\sigma^{-1}$ takes $\mathcal{N}$ to  $\mathcal{N}\cup\mathcal{S}$, where $\mathcal{N}$ is  a $w$-simple-minded system in $\mathcal{D}$. 
\end{Them}

\section{Reduction technique of simple-minded systems}
Starting from this section, we mainly consider  stable module category $A$-$\stmod$ of a self-injective algebra $A$ as a triangulated category. Note that $A$-$\stmod$  is a triangulated category with suspension functor $[1]$, where $[1]$ is the cosyzygy functor $\Omega^{-1}$. Its distinguished triangles are induced by short exact sequences in $A$-mod.  In this section, we study reduction of simple-minded systems.   From now on, we assume that $A$ is a self-injective $k$-algebra. We give a sufficient condition for $eAe$ to be a self-injective algebra, where $e$ is an idempotent element. 

\begin{Lem}{\rm(\cite[Chapter IV, Theorem  4.1]{SY})}\label{symmetric-algebra}
Let $A$ be a finite-dimensional  $k$-algebra.  If $A$ is a symmetric algebra, then for every idempotent $e$ of $A$, the algerba $eAe$ is symmetric.
\end{Lem}

\begin{Them}{\rm(\cite[Chapter IV, Theorem  6.1]{SY})}\label{self-injective-algebra}
Let $A$ be a finite-dimensional  $k$-algebra with $n_{A}$ simple modules. Then  $A$ is a self-injective algebra, if and only if there is a permutation $\nu$ on set   $\{ 1,\cdots\!,\, n_A\}$ such that $\top(Ae_i)\cong$ $\soc(Ae_{\nu(i)})$ on $A$-$\m$, where $i\in\{ 1,\cdots\!,\, n_A\}$.
\end{Them}

\begin{Cor}\label{eAe-self-injective}
Let $A$ be a finite-dimensional  self-injective algebra. Let  $E$ be a set of primitive idempotent elements and take  $e=\Sigma_{e_i\in E}\  e_i$. If  $E$ is  Nakayama-stable {\rm(}that is,  Nakayama functor permutes the simple modules which are corresponded by  primitive idempotent elements in $E${\rm)}, then  $eAe$  is a self-injective algebra.
\end{Cor}
\begin{proof}	
By Theorem \ref{self-injective-algebra}, we have $\top(Ae_j)\cong$ $\soc(Ae_{\nu(j)})$ on $A$-$\m$, where $e_{j}$ is a  primitive idempotent element and  $\nu$ a permutation on set   $\{ 1,\cdots\!,\, n_A\}$. Since $E$ is Nakayama-stable,  we have $e_{\nu(i)}\in E$, where $e_i\in E$.  Note that the unit of  $eAe$ is $e$, Nakayama functor $\D(A)\otimes_{A}-$ takes  simple module $S_i$ to be  simple module $S_{\nu(i)}$, where  $S_i$ is the simple module corresponding to  primitive idempotent element  $e_i$. 	It is known that  restriction functor $e(-)$ is an exact functor and $e(\add(Ae))\in\add(eAe)$ is a projective $eAe$-module.  Note that $e(-)$ does not preserve projective modules. By Theorem \ref{self-injective-algebra}, it suffices to show that for all $e_i\in E$, we have  $\top(eAe_i)\cong$ $\soc(eAe_{\nu(i)})$.

Take  a primitive idempotent element $e_i\in E$ and  projective cover $\pi_{i}\colon Ae_i\twoheadrightarrow \top(Ae_i)$.
Since restriction functor  $e(-)\colon$$A$-$\m\rightarrow eAe$-$\m$ is exact, it takes epimorphism  $\pi_{i}\colon Ae_i\twoheadrightarrow \top(Ae_i)$ to be epimorphism $e(\pi_{i})\colon eAe_i\twoheadrightarrow$ $e\top(Ae_i)$. Since $eAe_i$ is an indecomposable  projective  $eAe$-module and  $e\top(Ae_i)$ is a non-zero $eAe$-module, $e(\pi_{i})$ is a projective cover  of $e\top(Ae_i)$ in $eAe$-$\m$. Thus  $e\top(Ae_i)\cong \top(eAe_i)$.
Since  $e(-)$ is exact, it takes monomorphism  $\iota_{i}\colon\soc(Ae_{\nu(i)})\hookrightarrow Ae_{\nu(i)}$  to be monomorphism  $e(\iota_{i})\colon e\soc(Ae_{\nu(i)})\hookrightarrow eAe_{\nu(i)}$.
Since  $\D(e_iA)\cong Ae_{\nu(i)}$, $e\D(e_iA)\cong\D(e_iAe)$ and  $eAe_{\nu(i)}\cong\D(e_iAe)$ is an indecomposable injective $eAe$-module. Since $e\soc(Ae_{\nu(i)})$ is a non-zero simple $eAe$-module, we have $e\soc(Ae_{\nu(i)})\cong \soc(eAe_{\nu(i)})$.  Since $\top(Ae_j)\cong$ $\soc(Ae_{\nu(j)})$ for all $j$, 
we have $\top(eAe_i)\cong \soc(eAe_{\nu(i)})$  for all  $e_i\in E$.
\end{proof}

The following theorem is one  of our main results.

\begin{Them}\label{new-triangulated-category-2}
Let  $A$ be a  self-injective $k$-algebra and  let $\mathcal{S}=\{S_1,\cdots\!,\, S_n\}$ be the set of simple modules on $A$-$\m$. Let $\mathcal{S'}=\{S_1,\cdots\!,\, S_m\}$ be a Nakayama-stable subset of  $\mathcal{S}$.
Take $$\mathcal{D}={^{\bot}(\mathcal{S'})^{\bot}}=\{Z\in A{\text-}\stmod\, |\, \StHom_A(S,Z)=0, \    \StHom_A(Z,S)=0, S\in\mathcal{S'} \}.$$ Then  $\mathcal{D}$ is a triangulated category {\rm(}with a new structure{\rm)}, moreover, $\mathcal{D}$ and the stable module category  $eAe$-$\stmod$ of $eAe$ are equivalent as triangulated categories, where $e=e_{m+1}+\cdots+e_{n}$, $m<n$, $e_{i}$ is the primitive idempotent element corresponding to simple $A$-module $S_i$.	
\end{Them}

\begin{proof} 
We fix one notation as follows. For an object $X$	in $A$-$\m$, we have a unique (up to isomorphism) decomposition $X=X_{\mathcal{P}}\oplus X'$, where $X_{\mathcal{P}}$ has no nonzero projective summands and $X'$ is projective.
When we consider  $\underline{X}$ in $A$-$\stmod$, if there is no confusion, we still denote by $X$  the non-projective part $X_{\mathcal{P}}$. 
	
Note that, if  $A$ is a self-injective algebra, the  functor $\mathbb{S}_{-1}$ is  Nakayama functor $\D(A)\otimes_{A}-$. 
Since   $\mathcal{S'}$ is Nakayama-stable,  $\mathcal{S\backslash S'}$ is also Nakayama-stable.  Therefore  by Corollary \ref{eAe-self-injective},  $B:=eAe$ is a self-injective algebra, and by Theorem \ref{new-triangulated-category-1-1}, $\mathcal{D}$ is a triangulated category. The reader may refer to previous section for the triangulated  structure of $\mathcal{D}$.
	
We define a functor as follows. 
\begin{align*}
F\colon &\mathcal{D}\longrightarrow B{\text-}\stmod\\
&X\mapsto \left\{\begin{array}{ll} 0 & ( \ X\ is\ projective\ in\  A{\text-}\m); \\
eX & (\ X\ has\ no\ nonzero\ projective\ summands\ in\  A{\text-}\m).\end{array}\right.
\end{align*}

Let  $\StHom_A(X,Y)=\Hom_A(X,Y)/\mathcal{P}_{A}(X,Y)$ (resp.   $\StHom_B(M,N)=\Hom_B(M,N)/\mathcal{P}_{B}(M,N)$),  where  $\mathcal{P}_{A}(X,Y)$(resp.  $\mathcal{P}_{B}(M,N)$) is the subgroup of $\Hom_A(X,Y)$ (resp. $\Hom_B(M,N)$) factoring through a projective module.  We define the action of $F$ on morphism as follows.
For a given morphism $\underline{f}:X\rightarrow Y$ in $\mathcal{D}$.
If $X$ or  $Y$ is isomorphic to  zero in $\mathcal{D}$, then  $FX$ or  $FY$ is   isomorphic to  zero in $B$-$\stmod$. In this case we define $F(\underline{f})=0$. 
If both $X$ and $Y$ are not isomorphic to zero in $\mathcal{D}$. We know that every morphism in  $\mathcal{P}_{A}(X,Y)$ factors through projective cover $P(Y)$ of $Y$. Since $Y$ is in $\mathcal{D}$, $\top(Y)\in \add(\mathcal{S\backslash S'})$. Thus $P(Y)$ is in $\add(Ae)$ and $eP(Y)\in\add(eAe)$ is a  projective $B$-module, that is,  e$\mathcal{P}_{A}(X,Y)\subseteq\mathcal{P}_{B}(eX,eY)$ in $B$-$\m$. Hence for a morphism $\underline{f}=f+\mathcal{P}_{A}(X,Y):X\rightarrow Y$ in $\mathcal{D}$, we can define $F(\underline{f})$ to be  $\underline{ef}=ef+\mathcal{P}_{B}(eX,eY)$. Thus we know that $F(\underline{f})$ is well-defined and  $F$ is a functor.

\noindent {\bf Step one:} $F$ is a triangulated functor. We need to prove that there is a  functorially isomorphism $\varphi\colon F[1]'\cong [1]F$, and for any triangle  $\varepsilon$ in $\mathcal{D}$,  $F(\varepsilon)$ is still a triangle in $B$-$\stmod$, where $[1]'$ (resp. $[1]$) is the shift functor on $\mathcal{D}$ (resp. $B$-$\stmod$). 

First we prove $F[1]'\cong [1]F$.  If  $X$ is isomorphic to zero in $A$-$\stmod$, that is,  $X$ is a projective  $A$-module, then we have  $F[1]'(X)\cong [1]F(X)\cong 0$.  Given a non-zero $\underline{X}\in\mathcal{D}$, by  the definition of $[1]'$ (cf. \cite[Lemma 3.6]{CP}) in page 6, there is a triangle  as follows in $A$-$\stmod$:
\begin{equation}\label{ses-0}
\xymatrix{ \underline{s_X} \ar[r]^-{f_X} &  \underline{\Sigma X} \ar[r]^-{g} & \underline{X[1]'}\ar[r]^-{h} &  \underline{\Sigma s_X},}
	\end{equation}
where $f_X$ is a right minimal  $\mathcal{F(S')}$-approximation of $\underline{\Sigma X}$. Since $\mathcal{S'}$ is a set of simple module, each composition factor of a non-projective  module of $\mathcal{F(S')}$ in $A$-$\m$ is contained in $\mathcal{S'}$. Therefore each composition factor of $s_{X}$ is in $\mathcal{S'}$. 
Let \[\mathcal{X}:=\{M\mid  M\  \text{is a maxiaml indecomposable submodule of}\  \Sigma X\  \text{which is in}\ \mathcal{F(S')}\}.\] Take $s'_{X}=\bigoplus_{M\in\mathcal{X}}M$. Note that $s'_{X}$ is  a submodule of $\Sigma X$ in $A$-$\m$. It is routine to check that $s'_{X}\rightarrow\Sigma X$ in $A$-$\stmod$ induced by injective map is a right minimal $\mathcal{F(S')}$-approximation of $\Sigma X$. Therefore $s_{X}$ is isomorphic to $s'_{X}$ in $A$-$\stmod$ and $s_{X}$ may be chosen as a submodule of $\Sigma X$ in $A$-$\m$. Thus there is an injective map $f:s_{X}\hookrightarrow\Sigma X$ in $A$-$\m$. 
 Therefore there is an injective map $f:s_{X}\hookrightarrow\Sigma X$ in $A$-$\m$. Note that left minimal   $\mathcal{F(S')}$-approximation may be chosen dually. Hence there is a short exact sequence in $A$-$\m$ corresponding to triangle (\ref{ses-0}) as follows in $A$-$\stmod$:
\begin{equation}\label{ses-1} 
\xymatrix{0 \ar[r] & s_X \ar[r]^-{f} & \Sigma X \ar[r]^-{g} & X[1]' \ar[r] & 0.}
\end{equation} 
Applying functor $e(-)$ to (\ref{ses-1}), we have a short exact sequence in $B$-$\m$: 	
	$$\xymatrix{0 \ar[r] & e\Sigma X \ar[r]^-{eg} & eX[1]' \ar[r] & 0,}$$
therefore $e(\Sigma X) \cong e(X[1]')$ in $B$-$\m$, thus $\underline{e(\Sigma X)}\cong F(X[1]').$
	
Consider the following short exact sequence in $A$-$\m$:
\begin{equation}\label{ses-2} 
\xymatrix{0 \ar[r] & X \ar[r]^-{i} & I\ar[r]^-{j} & \Sigma X\ar[r] & 0,}
\end{equation}
where $i$ is the injective envelope of  $X$ and $I\in\add(Ae)$.  Applying  functor $e(-)$ to (\ref{ses-2}) in  $A$-$\m$, we have the following short exact sequence in $B$-$\m$: $$\xymatrix{0 \ar[r] & eX \ar[r]^-{ei} & eI\ar[r]^-{ej} & e\Sigma X\ar[r] & 0,}$$ 
It is easy to see that  $ei$ is  the injective envelope of $eX$ and  $\underline{e(\Sigma X)}\cong [1](\underline{eX})\cong  [1](FX)$. 
Thus  $F[1]'(X)\cong [1]F(X)$. It is routine to check naturality of $\varphi$, we omit the proof. Hence  $F[1]'\cong [1]F$ is a natural isomorphism.

Given a triangle  $$\varepsilon\colon\xymatrix{\underline{X} \ar[r]^-{\alpha} &  \underline{Y}\ar[r]^-{\beta} & \underline{Z}\ar[r]^-{\gamma} & \underline{X[1]'}}$$ in $\mathcal{D}$, we prove  $$F(\varepsilon)\colon\xymatrix{F(\underline{X}) \ar[r]^-{F\alpha} &  F(\underline{Y}) \ar[r]^-{F\beta} & F(\underline{Z})\ar[r]^-{F\gamma} & F(\underline{X}[1]') }$$ is a triangle  in $B$-$\stmod$. If  one of $\underline{X},\underline{Y}$ or $\underline{Z}$ is isomorphic to zero,  then it is clear that $F(\varepsilon)$ is a triangle. Therefore we may assume that none of   $\underline{X},\underline{Y}$ or $\underline{Z}$ are isomorphic to zero in  $A$-$\stmod$.  By the construction of distinguished  triangles in  $\mathcal{D}$, we have the following commutative diagram as follows.  
$$\xymatrix{
&   & \underline{s_f}\ar[d]_-{\alpha} \ar[r]^-{\sigma} & \underline{s_X} \ar[d]^-{\alpha_X}  \\
\underline{X} \ar[r]^-{f} & \underline{Y}\ar[dr]_-{g} \ar[r]^-{g_1} &\underline{Z_f} \ar[d]_-{\beta} \ar[r]^-{h_1} &  =\underline{\Sigma X}\ar[d]^-{\beta_X} & (\bigstar)\\
&   & \underline{Z} \ar[d]_-{\gamma} \ar[r]^-{h} & \underline{X}[1]'\ar[d]^-{\gamma_X}  \\	
& &  \underline{\Sigma s_f} \ar[r]_-{\Sigma\sigma} &  \underline{\Sigma s_X},\\
& & (\blacktriangle)}$$
where $\alpha$ (resp. $\alpha_X$) is a right minimal  $\mathcal{F(S')}$-approximation of $\underline{Z_f}$ (resp. $\Sigma X$). Since $(\bigstar)$ and $(\blacktriangle)$ are triangles in  $A$-$\stmod$, there are short exact sequences in  $A$-$\m$: 
\[\xymatrix{0 \ar[r] & X \ar[r]^-{} & Y\oplus P_1 \ar[r]^-{} & Z_f \ar[r] & 0}\]
	and 
\[ \xymatrix{0 \ar[r] & s_f \ar[r]^-{} & Z_f \ar[r]^-{} & Z \ar[r] & 0.}\]
Let $X\rightarrow I$ be the injective envelop of $X$. Then $\soc(X)\cong\soc(I)$ and  $P_{1}\in\add(I)$. Thus   $\soc(P_{1})\in\add(\soc(I))$. Since $X$ is in $\mathcal{D}$, any direct summand of  $\soc(X)$ is not in  $\mathcal{F(S')}$.  Since $\soc(I)\cong\soc(X)$, any direct summand of  $\soc(I)$ is not in  $\mathcal{F(S')}$. Therefore any direct summand of  $\soc(P_{1})$ is not in  $\mathcal{F(S')}$.      Since $S'$ is Nakayama-stable and $\top(P_{1})\cong\nu(\soc(P_{1}))$,  any direct summand of $\top(P_{1})$ is not in $\mathcal{F(S')}$. Hence $P_{1} \in \add(Ae)$. 
Applying functor  $e(-)$ to the above short exact sequences, we have  short exact sequences  in  $B$-$\m$:
$$\xymatrix{0 \ar[r] & eX \ar[r]^-{} & eY\oplus eP_1 \ar[r]^-{} & eZ_f \ar[r] & 0}$$
	and 
$$  \xymatrix{ 0\ar[r]^-{} & eZ_f\ar[r]^-{} & eZ \ar[r] & 0.}$$
where $eP_1\in\add(B)$. Thus it is easy to see that  $$\xymatrix{ F(\underline{X})\ar[r]^-{} & F(\underline{Y}) \ar[r]^-{} & F(\underline{Z})\ar[r]^-{} &  F(\underline{X})[1]'}$$ is a triangle in $B$-$\stmod$.

\noindent {\bf Step two:} $F$ is dense. We need to prove that for any non-zero $\underline{X}\in B$-$\stmod$, there is a $\underline{X'}\in\mathcal{D}$ such that $F(\underline{X'})\cong \underline{X}$. By Theorem  \ref{subset-of-sms}, there are two torsion pairs: $(^{\perp}\mathcal{S'}, \mathcal{F(S')})$ and $(\mathcal{F(S')},\mathcal{S'}^{\perp})$. Consider stable idempotent embedding functor $\underline{L_e}=\StHom_B(eA,-)\colon$ $B$-$\stmod\rightarrow A$-$\stmod$. Note that the composition $eL_e$ of idempotent embedding functor $L_e=\Hom_B(eA,-)\colon$ $B$-$\m\rightarrow A$-$\m$ and  restriction functor $eA\otimes_A-\colon$$A$-$\m\rightarrow B$-$\m$  is isomorphic to  $\Id_{B\text{-}\m}$.  Since $(\mathcal{F(S')},\mathcal{S'}^{\perp})$ is a torsion pair, there is a triangle in $A$-$\stmod$ as follows:
\begin{equation}\label{ses-3}
\xymatrix{  \underline{M}\ar[r]^-{\alpha_M} &  \underline{L_e}(\underline{X}) \ar[r]^-{} & \underline{X_1}\ar[r]^-{} &  \underline{M}[1],}
\end{equation}
where  $\alpha_M$ is a right minimal $\mathcal{F(S')}$-approximation of $\underline{L_e}(\underline{X})$  and  $\underline{X_1}\in  {\mathcal{S'}^{\perp}}$. By the construction of $\alpha_M$, we know that $\soc(X_{1})$ is not in $\mathcal{F(S')}$ in $A$-$\m$. Since $M$ is in $\mathcal{F(S')}$, each composition factor of $M$ is in $\mathcal{S'}$. Hence  $\Hom_A(M,X_1)=0$. Applying functor $e(-)$  to the short exact sequence in $A$-$\m$ corresponding to triangle  (\ref{ses-3}), we have 
\begin{equation}\label{ses-4-0}
X\cong eL_e(X)\cong e(X_1),
\end{equation}
where the first isomorphism  due to $eL_e\cong \text{id}_{B{\text-}\stmod}$.  Therefore $F(\underline{X})\cong F(\underline{X_{1}})$ in $B$-$\stmod$. Since  $(^{\perp}\mathcal{S'}, \mathcal{F(S')})$ is a torsion pair, there is a triangle as follows:
\begin{equation}\label{ses-4}
\xymatrix{  \underline{X'}\ar[r]^-{} &  \underline{X_1} \ar[r]^-{\beta_{M'}} & \underline{M'}\ar[r]^-{} &  \underline{X'}[1],}
\end{equation}
where $\beta_{M'}$ is a left minimal $\mathcal{F(S')}$-approximation of $\underline{X_1}$  and $\underline{X'}\in {^{\perp}\mathcal{S'}}$. 
Hence there exists a short exact sequence in $A$-$\m$ as follows:
\begin{equation}\label{ses-5}
\xymatrix{0 \ar[r] & X' \ar[r]^-{} & X_1\ar[r]^-{} & M' \ar[r] & 0.}
\end{equation}
	
We claim that $\underline{X'}\in \mathcal{D}$, $F(\underline{X'})\cong \underline{X}$. Since   $X'$ is a submodule of $X_1$ in $A$-$\m$ and for any  $\underline{M}\in \mathcal {F(S')}$, we have  $\Hom_A(M,X_1)=0$, thus $\Hom_A(M,X')=0$, then $\underline{X'}\in {\mathcal{S'}}^{\perp}$. Hence $\underline{X'}\in \mathcal{D}$. By applying functor $e(-)$ to the sequence  (\ref{ses-5}),  we have  $e(X')\cong e(X_1)$. By the sequence (\ref{ses-4-0}), we have  $\underline{X}\cong F(\underline{X_1})\cong F(\underline{X'})$ and  $\underline{X'}\in \mathcal{D}$.
	
\noindent {\bf Step three}: $F$ is full and faithful. 
First we prove that $F$ is full. It suffices to show that for any non-zero  $\underline{X}, \underline{Y}\in\mathcal{D}$, $F\colon \Hom_{\mathcal{D}}(\underline{X},\underline{Y})\rightarrow\StHom_{B}(F(\underline{X}),F(\underline{Y}))$ is an epimorphism.

{\bf Claim}: for any non-zero $\underline{X}, \underline{Y}\in\mathcal{D}$, if  $F(\underline{X})\cong F(\underline{Y})$, then $\underline{X}\cong\underline{Y}$. 
	
Since $\underline{X}, \underline{Y}\in\mathcal{D}$, the $\top(X)$, $\top(Y)$, $\soc(X)$ and $\soc(Y)$ belong to $\mathcal{F}(\mathcal{S}\backslash\mathcal{S'})$. Since $F(\underline{X})\cong F(\underline{Y})$, by the construction of functor $F$, we have $\top(X)\cong\top(Y)$, $\soc(X)\cong\soc(Y)$, and moreover we have  $X\cong Y$ in $A$-$\m$.  Thus $\underline{X}\cong\underline{Y}$ in $\mathcal{D}$. Note that this claim does not hold without the condition $\underline{X}, \underline{Y}\in\mathcal{D}$.

Take $f\in \StHom_{B}(F(\underline{X}),F(\underline{Y}))$. Since $(\mathcal{F(S')},\mathcal{S'}^{\perp})$ is a torsion pair, we have the following commutative diagram in $A$-$\stmod$:
\begin{equation}\label{com-diagram-1}\xymatrix{
M_{F(\underline{X})} \ar[d]_-{\alpha_f} \ar[r]^-{\alpha_{F(\underline{X})}} & \underline{L}_e(F(\underline{X})) \ar[d]_-{\underline{L}_e(f)} \ar[r]^-{g} & \underline{X_1}\ar[d]^-{\gamma} \ar[r] & \Sigma M_{F(\underline{X})} \ar[d]^-{\Sigma \alpha_f}\\
M_{F(\underline{Y})} \ar[r]_-{\alpha_{F(\underline{Y})}} & \underline{L}_e(F(\underline{Y}))  \ar[r]_-{g'} & \underline{Y_1} \ar[r] & \Sigma M_{F(\underline{Y})}\ .}
\end{equation}
Since $\alpha_{F(\underline{X})}$ is a right minimal  $\mathcal{F(S')}$-approximation,  $\alpha_f$ exists. Since $(^{\perp}\mathcal{S'}, \mathcal{F(S')})$ is a torsion pair, there is a commutative diagram as follows: 
\begin{equation}\label{com-diagram-2}
\xymatrix{
\underline{X'} \ar[d]_-{\sigma} \ar[r]^-{} & \underline{X_1}\ar[d]_-{\gamma} \ar[r]^-{\beta_{\underline{X_1}}} & M_{\underline{X_1}}\ar[d]^-{\delta} \ar[r] & \Sigma \underline{X'} \ar[d]^-{\Sigma \sigma}\\
\underline{Y'} \ar[r]_-{} & \underline{Y_1} \ar[r]_-{\beta_{\underline{Y_1}}} & M_{\underline{Y_1}} \ar[r] & \Sigma  \underline{Y'}\ .}
\end{equation}
Both $\underline{X'}$ and $\underline{Y'}$ are in $\mathcal{D}$.  Applying functor $F$  to the  triangles (\ref{com-diagram-1}) and  (\ref{com-diagram-2}), we have the following commutative diagrams:
\begin{equation}
\xymatrix{
			F(\underline{X})\ar[d]_-{f} \ar[r]^-{\cong}& \underline{eL_e}(F(\underline{X}))\ar[d]_-{	eL_e(f)} \ar[r]^-{\cong} & F(\underline{X_{1}})\ar[d]^-{e\gamma}  \\
			F(\underline{Y})\ar[r]_-{\cong}	& \underline{eL_e}(F(\underline{Y})) \ar[r]_-{\cong} & F(\underline{Y_1})\ ,
		} \ \ \ \ \ \ \ \ \ \ \ \ 
		\xymatrix{
			F(\underline{X'})\ar[d]_-{e\sigma} \ar[r]^-{\cong} & F(\underline{X_{1}})\ar[d]^-{e\gamma}  \\
		F(\underline{Y'})\ar[r]_-{\cong} & F(\underline{Y_{1}}) \ .}
\end{equation}
Therefore $\StHom_{B}(F(\underline{X'}),F(\underline{Y'}))\cong\StHom_{B}(F(\underline{X}),F(\underline{Y}))$. Since both $X$ and $X'$ (resp. $Y$ and  $Y'$) are in $\mathcal{D}$, and  $F(\underline{X})\cong F(\underline{X}')$ (resp. $F(\underline{Y})\cong F(\underline{Y'})$), we have  $\underline{X}\cong \underline{X'}$ (resp. $\underline{Y}\cong \underline{Y'}$) in $\mathcal{D}$. Hence for morphism $f\in\Hom_{B}(F(\underline{X}),F(\underline{Y}))$, there is a morphism $f'\colon \underline{X}\rightarrow \underline{Y}$ such that  $ef'=f$ up to isomorphism.
	
Finally we prove that $F$ is faithful.  Since  $F$ is full, by \cite[p. 446]{Ric}, it suffices to show that for any non-zero object $\underline{X}$ in $\mathcal{D}$, $F(\underline{X})$ is non-zero. Since  $\top(X)$ is not a projective $A$-module, $\underline{\top(X)}$ is not in $\mathcal{F}(\mathcal{S'})$, we know that $F(\underline{\top(X)})$ is non-zero in $B$-$\m$. Thus  $F(\underline{X})$ is non-zero in  $B$-$\stmod$.
\end{proof}

\begin{Rem}
We recall some facts on restriction functor $e(-)=eA\otimes_{A}-$. Restriction functor is an exact functor, however it does not preserve projective modules. Therefore restriction functor is not  induced to be the standard stable functor.  $e(\add(Ae))\in\add(B)$ is a projective $B$-module, $S'_i:=e(S_i)$ is the simple $B$-module corresponding to the primitive idempotent element $e_{i}$, where  $i=m+1,\cdots\!,\, n$.  We have  $e(M)=0$ for any  $M\in\mathcal{F(S')}$. 
\end{Rem}

\begin{Cor}\label{new-triangulated-category-3}
	Let  $A$ be a symmetric $k$-algebra and  let $\mathcal{S}=\{S_1,\cdots\!,\, S_n\}$ be the set of simple modules on $A$-$\m$. Let $\mathcal{S'}$ be a subset of   $\mathcal{S}$.
	Take $$\mathcal{D}={^{\bot}(\mathcal{S'})^{\bot}}=\{Z\in A{\text-}\stmod\, |\, \StHom_A(S,Z)=0, \    \StHom_A(Z,S)=0, S\in\mathcal{S'} \}.$$ Then  $\mathcal{D}$ is a triangulated category {\rm(}with a new structure{\rm)}, moreover, $\mathcal{D}$ and the stable module category  $eAe$-$\stmod$ of $eAe$ are equivalent as triangulated categories, where $e=\Sigma_{e_{i}\in\mathcal{S}'}e_{i}$, where $e_{i}$ is the primitive idempotent element corresponding to simple $A$-module $S_i$.	
\end{Cor}

\begin{Ex}\label{BGA-reduction}
We present an explicit  example for Theorem \ref{new-triangulated-category-2}. Consider Brauer graph algebra $A$ with Brauer graph given by
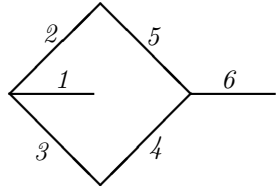
\begin{center}
	\vspace{.5cm}
	\setlength{\unitlength}{0.8cm}
	\begin{picture}(6,5)
		\thicklines
		\put(-0.5,3.5){\line(1,1){1.5}}
		\put(-0.5,3.5){\line(1,0){1.4}}
		\put(-0.5,3.5){\line(1,-1){1.5}}
		\put(1,5){\line(1,-1){1.5}}
		\put(2.5,3.5){\line(1,0){1.4}} 
		\put(2.5,3.5){\line(-1,-1){1.5}}
		\put(0.25,3.6){1}
		\put(0.05,4.35){2}
		\put(-0.1,2.4){3}
		\put(1.8,2.5){4}
		\put(1.75,4.3){5}
		\put(3,3.6){6}
	\end{picture}.
\end{center}
\noindent Then $A=~~\begin{matrix}1\\2\\3\\1\end{matrix}
~~\oplus~~\begin{matrix}2\\\begin{matrix}3\\1\end{matrix}~~\begin{matrix}5\end{matrix}
	\\2\end{matrix}
~~\oplus ~~\begin{matrix}3\\\begin{matrix}1\\2\end{matrix}~~\begin{matrix}4\end{matrix}
	\\3\end{matrix}
~~\oplus~~\begin{matrix}4\\\begin{matrix}3\end{matrix}~~\begin{matrix}6\\5\end{matrix}\\4
\end{matrix}
~~\oplus~~\begin{matrix}5\\\begin{matrix}2\end{matrix}~~\begin{matrix}4\\6
	\end{matrix}\\5\end{matrix}
~~\oplus~~\begin{matrix}6\\5\\4\\6\end{matrix}.$
The corresponding quiver $Q$ of $A$ is the following diagram:
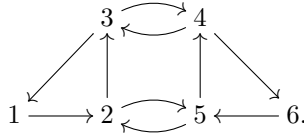
\[\xymatrix{
&  3\ar[dl]^{} \ar@/^/[r]  & 4\ar@/^/[l]  \ar[dr]^{} &  \\
 	1\ar[r]^{} & 2\ar[u]^{}  \ar@/^/[r] & 5\ar[u]_{}\ar@/^/[l] & 6 \ar[l]^{}.}
\]

\bigskip

The set of simple module is $\mathcal{S}=\{S_{1},S_{2},S_{3},S_{4},S_{5},S_{6}\}$. Take $\mathcal{S}'_{1}=\{S_{1},S_{2},S_{3}\}$. Since $A$ is a symmetric algebra, $\mathcal{S}'_{1}$ is clear a Nakayama-stable subset of $\mathcal{S}$. In this case, 
\[\mathcal{D}_{1}=\{Z\in A{\text-}\stmod\, |\, \StHom_A(\mathcal{S}'_{1},Z)=0, \    \StHom_A(Z,\mathcal{S}'_{1})=0 \}=\{4,~~5,~~6,~~\begin{matrix}4\\5\end{matrix},~~ \begin{matrix}5\\6\end{matrix},~~\begin{matrix}4\\6\end{matrix},~~\begin{matrix}4\\6\\5\end{matrix},~~\begin{matrix}5\\4\\6\end{matrix},~~\begin{matrix}6\\5\\4\end{matrix}\}.\] It is isomorphic to the stable module category of the algebra $eAe=~~\begin{matrix}4\\6\\5\\4\end{matrix}
~~\oplus~~\begin{matrix}5\\4\\6\\5\end{matrix}~~\oplus~~\begin{matrix}6\\5\\4\\6\end{matrix}$, where $e=e_{4}+e_{5}+e_{6}$. The quiver $Q_{1}$ corresponding to the algebra  $eAe$ is as follows {\rm(}we omit the corresponding ideal $I_{1}${\rm)}.
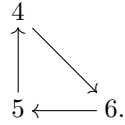
\[\xymatrix{
	  4 \ar[dr]^{} &  \\
 5\ar[u]_{} & 6 \ar[l]^{}.}
\]

Take $\mathcal{S}'_{2}=\{S_{1}\}$. $\mathcal{S}'_{2}$ is also a Nakayama-stable subset of $\mathcal{S}$. In this case, 
\[\mathcal{D}_{2}=\{Z\in A{\text-}\stmod\, |\, \StHom_A(\mathcal{S}'_{2},Z)=0, \    \StHom_A(Z,\mathcal{S}'_{2})=0\}.\] $\mathcal{D}_{2}$ is an infinite set. It is isomorphic to the stable module category of the algebra \[eAe=~~\begin{matrix}2\\\begin{matrix}3\end{matrix}~~\begin{matrix}5\end{matrix}
	\\2\end{matrix}
\oplus ~\begin{matrix}3\\\begin{matrix}2\end{matrix}~~\begin{matrix}4\end{matrix}
	\\3\end{matrix}
\oplus~\begin{matrix}4\\\begin{matrix}3\end{matrix}~~\begin{matrix}6\\5\end{matrix}\\4
\end{matrix}
\oplus~\begin{matrix}5\\\begin{matrix}2\end{matrix}~~\begin{matrix}4\\6
	\end{matrix}\\5\end{matrix}
\oplus~\begin{matrix}6\\5\\4\\6\end{matrix},\] where $e=e_{2}+e_{3}+e_{4}+e_{5}+e_{6}$. The quiver $Q_{2}$ corresponding to the algebra  $eAe$ is as follows {\rm(}we omit the corresponding ideal  $I_{2}${\rm)}.
\[\xymatrix{
	&  3\ar@/^/[r]\ar@/^/[d]  & 4\ar@/^/[l]  \ar[dr]^{} &  \\
	& 2\ar@/^/[u]  \ar@/^/[r] & 5\ar[u]_{}\ar@/^/[l] & 6 \ar[l]^{}.}
\]
\end{Ex}

\begin{Rem}\label{hu-xi-simple}
\begin{enumerate}[$(1)$]
\item $\mathcal{D}$ in Theorem \ref{new-triangulated-category-2} is not a triangulated subcategory of $A$-$\stmod$ and its new triangle structure is given in Section 2.
\item The $\mathcal{D}_{1}$ in Example \ref{BGA-reduction} is the stable module category of the algebra $kQ_{1}/I_{1}$. The form of objects in $\mathcal{D}_{2}$ is different from the objects in stable module category $kQ_{2}/I_{2}$-$\stmod$, however 
$\mathcal{D}_{2}$ with the new triangulated structure is isomorphic to  $kQ_{2}/I_{2}$-$\stmod$ with usual triangulated structure as the stable module category of a self-injective algebra.
\end{enumerate}
\end{Rem}

By configuration theory developed by Riedtmann {\rm(}See for example \cite{Rie1, Rie3}{\rm)}, Hu and Xi proved the following result. 
\begin{Prop}{\rm(\cite[Lemma  5.6]{HX})}\label{RFS-stable-simple} 
 Let $A$ and $B$ be two representation-finite, self-injective $k$-algebras without nonzero semisimple direct summands. Suppose that $\Phi\colon A$-$\stmod\rightarrow B$-$\stmod$ is a stable equivalence of Morita type. Then there are a simple  $A$-module $X$ and integers $r$ and $t$ such that $\tau^{r}\circ\Omega^{t} \circ \Phi(X)$ is isomorphic in $B$-mod to a simple $B$-module, where $\tau$ and $\Omega$ stand for the Auslander–Reiten translation and syzygy functor, respectively.
\end{Prop}

From the proof of Hu and  Xi, we replace simple module $X$ by an object of a simple-minded system, the above result also holds.   
\begin{Cor}\label{RFS-stable-simple-sms} 
Let $A$ and $B$ be two representation-finite, self-injective $k$-algebras without nonzero semisimple direct summands and let $\mathcal{S}$ be a simple-minded system in $A$-$\stmod$. Suppose that $\Phi\colon A$-$\stmod\rightarrow B$-$\stmod$ is a stable equivalence of Morita type. Then there are an object  $S\in\mathcal{S}$  and integers $r$ and $t$ such that $\tau^{r}\Omega^{t}\Phi(S)$ is isomorphic in $B$-mod to a simple $B$-module, where $\tau$ and $\Omega$ stand for the Auslander–Reiten translation and syzygy functor, respectively.
\end{Cor}

We provide a new proof for the following result by Theorem \ref{new-triangulated-category-2}.
\begin{Prop}\label{RFS-reduction} 
Let $A$ be a  representation-finite symmetric algebra and $\mathcal{S}$ a simple-minded system in $A$-$\stmod$. Then  the cardinality of $\mathcal{S}$ is the number of non-isomorphic, non-projective simple $A$-modules.
\end{Prop}
\begin{proof}
We induct on the number of non-isomorphic, non-projective simple $A$-modules.
First, the local  representation-finite self-injective algebras are precise the local symmetric Nakayama algebras.  We know that the cardinality of a  simple-minded system of a local symmetric Nakayama algebra is one, thus our conclusion holds for $k=1$. Now we assume our conclusion holds for symmetric algebras with $k=n-1$ non-isomorphic,non-projective simple $A$-modules. We prove the case for  $k=n$.

Let  $\mathcal{S}$ be a simple-minded system in $A$-$\stmod$. By Corollary \ref{RFS-stable-simple-sms},   
there are  $S\in\mathcal{S}$ and integers $r,\,t$ such that $\alpha(S):=\tau^{r}\Omega^{t}(S)$  is a simple $A$-module. Note that $\alpha:A$-$\stmod\rightarrow A$-$\stmod$ is a stable self-equivalence. Since  simple-minded systems are invariant under stable equivalences,  $\alpha(\mathcal{S})$ is a simple-minded system in  $A$-$\stmod$ and moreover  $|\mathcal{S}|=|\alpha(\mathcal{S})|$. Let  $e$ be the primitive idempotent element corresponding to $A$-module $\alpha(S)$ and take  $A'=(1-e)A(1-e)$. 
By Lemma \ref{symmetric-algebra},  $A'$ is a  representation-finite symmetric  algebra  and it has  $n-1$ non-isomorphic, non-projective simple $A$-modules.  By induction, the cardinality of a simple-minded system of $A'$ is $n-1$.  Let $\mathcal{D}$ be  the stable bi-perpendicular full subcategory of $\alpha(S)$ on $A$-$\stmod$.  By Theorem \ref{new-triangulated-category-2}, there is an equivalence $F:\mathcal{D}\rightarrow A'$-$\stmod$ as triangulated categories. By Theorem  \ref{sms-1-1-corresponding-1}, there are one to one correspondences as follows:
\vspace{-.1cm}
\[\xymatrix{
\{sms's\ of\ A{\text-}\stmod\ containing \ S\}  \ar@{<->}[d]^-{\alpha}  &   \\
\{sms's\ of\ A{\text-}\stmod\ containing\ \alpha(S)\} \ar@{<->}[d]^-{\beta}& \\
\{sms's\ of\ \mathcal{D}\} \ar@{<->}[d]^-{F} & \\
\{sms's\ of \  A'{\text-}\stmod\}, 
}\]
	
\noindent where $\beta$ is the bijection in Theorem \ref{sms-1-1-corresponding-1}, it maps  simple-minded system $\mathcal{S}$ on $A$-$\stmod$ to simple-minded system $\alpha(\mathcal{S})\backslash\{\alpha(S)\}$ on   $\mathcal{D}$. 
Hence   $F(\beta(\alpha(\mathcal{S})))= F(\alpha(\mathcal{S})\backslash\{\alpha(S)\})$ is a simple-minded system on $A'$-$\stmod$. By induction,  the cardinality of $F(\alpha(\mathcal{S})\backslash\{\alpha(S)\})$ is $n-1$. 
Since both $\alpha$ and $F$  are stable equivalences, the cardinality of  $\mathcal{S}$ is $n=(n-1)+1$, thus our induction is finished.
\end{proof}

\section{Gluing technique of simple-minded systems}
In this section, we construct recollements through simple-minded systems in the stable module category $A$-$\stmod$ of  a self-injective algebra $A$ and we study gluing technique of simple-minded systems. We also study extendible property of a simple-minded system in $A$-$\stmod$.
\subsection{Gluing technique of simple-minded systems} Let $A$ be a self-injective algebra. The condition $\Omega(\mathcal{S})\subseteq\mathcal{S}$ is a necessary condition for a family  $\mathcal{S}$ of objects  to be a simple-minded system in $A$-$\stmod$.  We state some results for this necessary condition as follows.
\begin{Lem} $($\cite[Lemma  4.6 and Lemma 4.7]{Z}$)$\label{Omega-F(S)}
Let $A$ be a self-injective algebra and $\mathcal{S}$  an  orthogonal system  in $A$-$\stmod$.  Then $\Omega(\mathcal{S})$ {\rm(}resp. $\Omega^{-1}(\mathcal{S})${\rm)} is contained in  $\mathcal{F}(\mathcal{S})$ if and only if $\Omega(\mathcal{F}(\mathcal{S}))$  {\rm(}resp. $\Omega^{-1}(\mathcal{F}(\mathcal{S}))${\rm)} is contained in  $\mathcal{F}(\mathcal{S})$.
\end{Lem}

The following lemma is  a little unexpected and we include a  proof here. 
\begin{Lem}$($\cite[Lemma  4.8]{Z}$)$\label{omega-self-injective-1}
Let $A$ be a self-injective algebra and let $\mathcal{S}=\{S_{i}\mid i\in I\}$ be a Nakayama-stable {\rm(}that is, $\nu(S)=S${\rm)} orthogonal system in $A$-$\stmod$.  Then $\Omega(\mathcal{S})$ is contained in $\mathcal{F}(\mathcal{S})$ if and only if $\Omega^{-1}(\mathcal{S})$ is contained in  $\mathcal{F}(\mathcal{S})$. 	
\end{Lem}	
\begin{proof}
We prove only that if $\Omega(\mathcal{S})$ is contained in $\mathcal{F}(\mathcal{S})$, then $\Omega^{-1}(\mathcal{S})$ is also contained in  $\mathcal{F}(\mathcal{S})$. The converse part is dual. Take  $S_{j}$ in $\mathcal{S}$ for $j\in I$.  Since $\Omega(\mathcal{S})$ is contained in $\mathcal{F}(\mathcal{S})$, there is  a positive integer $m$  such that $\Omega(S_{j})$ is in some $ (\mathcal{S})_{m}$. Since $\Omega(S_{j})\in (\mathcal{S})_{m}$,  there is a non-split triangle as follows.
	
\begin{equation}\label{seq-1}
	S_{i}\xrightarrow{\alpha}\Omega(S_{j})\xrightarrow{}M\xrightarrow{} S_{i}[1],
\end{equation}
\noindent where $S_{i}\in\mathcal{S}$ and $M\in (\mathcal{S})_{m-1}$.
	
Claim that $S_{i}=\nu^{-1}(S_{j})$. Otherwise $\StHom_A(S_{i}, \Omega(S_{j}))\cong\D\StHom_A(\Omega(S_{j}), \nu\Omega(S_{i}))\cong\StHom_A(S_{j}, \nu(S_{i}))\\\cong 0,$ thus triangle (\ref{seq-1}) splits, it is a contradiction. 
Rotating the above triangle (\ref{seq-1}) to the right twice,  we have 
\[M\xrightarrow{\alpha}\nu^{-1}(S_{j})[1]\xrightarrow{}S_{j}\xrightarrow{} M[1].\] 	
Since  $M\in (\mathcal{S})_{m-1}$ and $S_{j}\in\mathcal{S}$, $\Omega^{-1}(\nu^{-1}(S_{j}))=\nu^{-1}(S_{j})[1]$ is in $(\mathcal{S})_{m}$. 
Since $\nu$  is an self-equivalence of $A$-$\stmod$ and it is  also a permutation on $\mathcal{S}$,  it is clear that $\{\Omega^{-1}(\nu^{-1}(S_{j}))\mid j\in I\}=\Omega^{-1}(\mathcal{S})$. Hence 
$\Omega^{-1}(\mathcal{S})$ is contained in  $\mathcal{F}(\mathcal{S})$.
\end{proof}

\begin{Prop}\label{omega-tri-subcat}
Let $A$ be a self-injective algebra and let $\mathcal{S}$ be a Nakayama-stable orthogonal system in $A$-$\stmod$.  If  $\Omega(\mathcal{S})$ is contained in $\mathcal{F}(\mathcal{S})$, then $\mathcal{F}(\mathcal{S})$ is a triangulated subcategory of $A$-$\stmod$. 	
\end{Prop}	
\begin{proof}
It is a direct consequence of Lemma \ref{omega-self-injective-1} and definition of a triangulated subbcategory. Note that $\mathcal{F}(\mathcal{S})$  is the smallest extension-closed subcategory containing $\mathcal{S}$.
\end{proof}

\begin{Cor}\label{omega-tri-subcat-1}
Let $A$ be a self-injective algebra and let $\mathcal{S}$ be a Nakayama-stable orthogonal system in $A$-$\stmod$.  If   $\Omega(\mathcal{S})$ is contained in $\mathcal{F}(\mathcal{S})$, then   $^{\perp}\mathcal{S}$ and  $ \mathcal{S}^{\perp}$ are triangulated subcategories. In particular, the stable bi-perpendicular category  $^{\perp}\mathcal{S}^{\perp}$ is a triangulated subcategory of $A$-$\stmod$. 	
\end{Cor}	
\begin{proof}
We only prove that $\mathcal{S}^{\perp}$  is a triangulated subcategory, since the case for $^{\perp}\mathcal{S}$ is similar to prove. Note that, for  an object $M\in A$-$\stmod$, $M\in\mathcal{S}^{\perp}$ if and only if  $M\in\mathcal{F}(\mathcal{S})^{\perp}$.  Take  $M\in\mathcal{S}^{\perp}[i]$ for any integer $i$.  We have $M[-i]\in\mathcal{S}^{\perp}$,  therefore  $\StHom_A(S,M[-i])\cong \StHom_A(S[i],M)=0$ for any $S\in\mathcal{F}(\mathcal{S})$ and $i\in\mathbb{Z}$.  By Proposition \ref{omega-tri-subcat}, $\mathcal{F}(\mathcal{S})$ is a triangulated subcategory, then $\mathcal{F}(\mathcal{S})[i]=\mathcal{F}(\mathcal{S})$. Therefore $M\in\mathcal{F}(\mathcal{S})^{\perp}=\mathcal{S}^{\perp}$. Thus $\mathcal{S}^{\perp}[i]\subseteq\mathcal{S}^{\perp}$.  Thus $\mathcal{S}^{\perp}$  is closed under shift functor $[i]$ for any integer $i$. Since $\mathcal{S}^{\perp}$ is  also closed under extension, by definition, it is a  triangulated subcategory.
\end{proof}

\begin{Lem}\label{omega-tri-subcat-0}
Let $A$ be a self-injective algebra and let $\mathcal{S}$ be a Nakayama-stable orthogonal system in $A$-$\stmod$.  If  $\Omega(\mathcal{S})$ is contained in $\mathcal{F}(\mathcal{S})$ and $\mathcal{F}(\mathcal{S})$ is contravariantly finite {\rm(}resp. covariantly finite{\rm)} in $A$-$\stmod$, then  $^{\perp}\mathcal{S}={^{\perp}\mathcal{S}^{\perp}}$ {\rm(}resp. ${\mathcal{S}^{\perp}}={^{\perp}\mathcal{S}^{\perp}}${\rm)}.
\end{Lem}	
\begin{proof}
We only prove $^{\perp}\mathcal{S}={^{\perp}\mathcal{S}^{\perp}}$, since the case ${^{\perp}\mathcal{S}^{\perp}}=\mathcal{S}^{\perp}$ can be proved dually. It is clear that ${^{\perp}\mathcal{S}^{\perp}}\subseteq{^{\perp}\mathcal{S}}$. It suffices to show $^{\perp}\mathcal{S}\subseteq{^{\perp}\mathcal{S}^{\perp}}$. Since  $\mathcal{F}(\mathcal{S})$ is contravariantly finite in $A$-$\stmod$, by Theorem \ref{two-torsion-pairs}, $(\mathcal{F}(\mathcal{S}),\mathcal{S}^{\bot})$ is a torsion pair  in $A$-$\stmod$. Take a non-zero indecomposable object $X\in {^{\perp}\mathcal{S}}$ and consider  the minimal triangle of $X$ via torsion pair $(\mathcal{F}(\mathcal{S}),\mathcal{S}^{\bot})$ as follows.
\[\xymatrix{{\bf c}X\ar[r]^{f} &X\ar[r]^{g} &{\bf d}X\ar[r]^{h}& {\bf c}X[1].}\]
Since $\mathcal{S}$ is  Nakayama-stable,  by Lemma \ref{subset-of-sms-bipen} and Remark \ref{bi-per-contra}, 
${\bf d}X\in{^{\perp}\mathcal{S}^{\perp}}$. Since ${\bf c}X\in\mathcal{F}(\mathcal{S})$ and $\mathcal{F}(\mathcal{S})$ is a triangulated subcategory,  ${\bf c}X[1]\in\mathcal{F}(\mathcal{S})$. Thus $h=0,$ $g$ is a split epimorphism and $f$ is a split monomorphism. Since $X$ is indecomposable, ${\bf c}X\cong X$ or ${\bf c}X\cong 0$. Since $\mathcal{F}(\mathcal{S})\cap{^{\perp}\mathcal{S}}=\{0\}$, ${\bf c}X=0$. Thus $g$ is an isomorphism and $X\cong{\bf d}X \in{^{\perp}\mathcal{S}^{\perp}}$. Hence  $^{\perp}\mathcal{S}\subseteq{^{\perp}\mathcal{S}^{\perp}}$. 
\end{proof}

\begin{Lem}\label{splitting- torsion-pair}
Let $A$ be a self-injective algebra and let $\mathcal{S}$ be a Nakayama-stable orthogonal system in $A$-$\stmod$.  If  $\Omega(\mathcal{S})$ is contained in $\mathcal{F}(\mathcal{S})$ and $\mathcal{F}(\mathcal{S})$ is functorically finite in $A$-$\stmod$, then 
torsion pairs $(^{\bot}{\mathcal{S}},\mathcal{F}(\mathcal{S}))$ and  $(\mathcal{F}(\mathcal{S}),\mathcal{S}^{\bot})$ are splitting, that is, for any indecomposble object $X$ in $A$-$\stmod$, $X$ is either in $\mathcal{F}(\mathcal{S})$ or in $^{\bot}{\mathcal{S}}=\mathcal{S}^{\bot}$.
\end{Lem}
\begin{proof}
By  Theorem \ref{two-torsion-pairs}, there are two torsion pairs $(^{\bot}{\mathcal{S}},\mathcal{F}(\mathcal{S}))$ and  $(\mathcal{F}(\mathcal{S}),\mathcal{S}^{\bot})$ in $A$-$\stmod$. By Lemma \ref{omega-tri-subcat-0}, we have $^{\perp}\mathcal{S}={^{\perp}\mathcal{S}^{\perp}}= {\mathcal{S}^{\perp}}$.
We prove only that  torsion pair   $(\mathcal{F}(\mathcal{S}),\mathcal{S}^{\bot})$ is splitting, since the other one can be handled dually. For a non-zero indecomposable object $X$ in $A$-$\stmod$, there is a triangle as follows.
\[\xymatrix{{\bf c}X\ar[r]^{f} &X\ar[r]^{g} &{\bf d}X\ar[r]^{h}& {\bf c}X[1],}\] 
where ${\bf c}X\in\mathcal{F}(\mathcal{S})$ and ${\bf d}X\in\mathcal{S}^{\bot}$. Since  $\mathcal{F}(\mathcal{S})$ is triangulated, ${\bf c}X[1]\in\mathcal{F}(\mathcal{S})$. By  Lemma \ref{omega-tri-subcat-0},
${\bf d}X\in\mathcal{S}^{\bot}={^{\perp}\mathcal{S}^{\perp}}$. Thus $h=0$. Then $g$ is a split epimorphism. Since $X$ is non-zero and indecomposable, ${\bf d}X$ is isomorphic to zero or $g$ is an isomorphism. Hence $X\cong{\bf c}X\in\mathcal{F}(\mathcal{S})$ or $X\cong{\bf d}X\in\mathcal{S}^{\bot}$. 
\end{proof}

\begin{Cor}\label{triangulated-category-2}
Let  $A$ be a  self-injective $k$-algebra and  let $\mathcal{S}=\{S_1,\cdots\!,\, S_n\}$ be the set of simple modules on $A$-$\m$. Let $\mathcal{S'}=\{S_1,\cdots\!,\, S_m\}$ be a Nakayama-stable subset of  $\mathcal{S}$ such that $\Omega(\mathcal{S'})\subseteq\mathcal{F}(\mathcal{S'})$.
Take $$\mathcal{D}={^{\bot}(\mathcal{S'})^{\bot}}=\{Z\in A{\text-}\stmod\, |\, \StHom_A(S,Z)=0, \    \StHom_A(Z,S)=0, S\in\mathcal{S'} \}.$$ 
Then $\mathcal{D}$ is a triangulated category of $A$-$\stmod$. Moreover $\mathcal{D}$ and the stable module category  $eAe$-$\stmod$ of $eAe$ are equivalent as triangulated categories, where $e=e_{m+1}+\cdots+e_{n}$, $m<n$, $e_{i}$ is the primitive idempotent element corresponding to simple $A$-module $S_i$.	
\end{Cor}
\begin{proof}
By Corollary \ref{omega-tri-subcat-1}, $\mathcal{D}$ is a triangulated subcategory of $A$-$\stmod$. It is not hard to see that, if $\Omega(\mathcal{S'})\subseteq\mathcal{F}(\mathcal{S'})$, then the triangulated structure of $\mathcal{D}$ given in  Section 2 is the same with $\mathcal{D}$ as a triangulated subcategory of $A$-$\stmod$. It follows from Theorem \ref{new-triangulated-category-2} that $\mathcal{D}$ is equivalent to $eAe$-$\stmod$ as triangulated categories. 
\end{proof}

\begin{Them}\label{recollement-sms}
Let $A$ be a self-injective algebra and $\mathcal{S}$ a Nakayama-stable orthogonal system in $A$-$\stmod$. If  $\mathcal{F}(\mathcal{S})$ is functorially finite in $A$-$\stmod$ such that  $\Omega(\mathcal{S})$ is contained in $\mathcal{F}(\mathcal{S})$,  then 
\begin{enumerate}[$(1)$]
\item  $A$-$\stmod$ is a recollement of $\mathcal{F}(\mathcal{S})$ and $\mathcal{S}^{\bot}$ as follows.
\begin{align}\label{recoll-2}
\xymatrix{\mathcal{F}(\mathcal{S})\ar^-{i_*=i_!}[rr]&&A{\text-}\stmod\ar^-{j^!=j^*}[rr]\ar^-{i^!}@/^1.2pc/[ll]\ar_-{i^*}@/_1.6pc/[ll]&&\mathcal{S}^{\bot}\ar^-{j_*}@/^1.2pc/[ll]\ar_-{j_!}@/_1.6pc/[ll]}.
\end{align}
\item Let $\mathcal{M}=\{M_{i}\mid i=1,\cdots,m\}$ be a simple-minded system of $\mathcal{F}(\mathcal{S})$ and $\mathcal{N}=\{N_{j}\mid i=1,\cdots,n\}$ a simple-minded system of $\mathcal{S}^{\bot}$.  Then $\mathcal{S}'=\{i_{\ast}(M)\mid M\in\mathcal{M}\}\cup\{j_{\ast}(N)\mid N\in\mathcal{N}\}\}$ a simple-minded system of $A$-$\stmod$.
\end{enumerate}

\end{Them}	
\begin{proof}
\noindent(1) Since  $\mathcal{F}(\mathcal{S})$ is functorially finite in $A$-$\stmod$, by Theorem \ref{two-torsion-pairs}, there are two torsion pairs $(^{\bot}{\mathcal{S}},\mathcal{F}(\mathcal{S}))$ and  $(\mathcal{F}(\mathcal{S}),\mathcal{S}^{\bot})$ in $A$-$\stmod$. Since $\Omega(\mathcal{S})$ is contained in $\mathcal{F}(\mathcal{S})$ and $\mathcal{S}$ is Nakayama-stable, by Proposition \ref{omega-tri-subcat}, $\mathcal{F}(\mathcal{S})$  is a triangulated subcategory, thus $\mathcal{F}(\mathcal{S})$ is closed under shift functor $[i]$ for any integer $i$.  By  Corollary \ref{omega-tri-subcat-1},  $\mathcal{S}^{\bot}$ is a triangulated subcategory, thus $\mathcal{S}^{\bot}$ is closed under shift functor $[i]$ for any integer $i$. It follows from definition that the pair  $(\mathcal{F}(\mathcal{S}),\mathcal{S}^{\bot})=(\mathcal{F}(\mathcal{S}),\mathcal{S}^{\bot}[1])$ is a t-structure.  Similarly, $(^{\bot}{\mathcal{S}},\mathcal{F}(\mathcal{S})) =(^{\bot}{\mathcal{S}},\mathcal{F}(\mathcal{S})[1])$ is also a t-structure. 
Hence we have two t-structures $(^{\bot}{\mathcal{S}},\mathcal{F}(\mathcal{S}))$ and $ (\mathcal{F}(\mathcal{S}),\mathcal{S}^{\bot})$. It is clear from  Proposition \ref{omega-tri-subcat} that  the three subcategories $^{\bot}{\mathcal{S}}$, $\mathcal{F}(\mathcal{S})$ and $\mathcal{S}^{\bot}$ are triangulated subcategories. Thus  $(^{\bot}{\mathcal{S}},\mathcal{F}(\mathcal{S}))$ and $ (\mathcal{F}(\mathcal{S}),\mathcal{S}^{\bot})$ are stable t-structures. By Theorem  \ref{TTF-triple-recollement}, there is a recollement (\ref{recoll-2}) as the above theorem. Note that the six functors in  (\ref{recoll-2}) is described in Theorem  \ref{TTF-triple-recollement}, and that $i_{*}$ and $j_{*}$ are embedding functors. 

\medskip

\noindent(2) Let $\mathcal{M}=\{M_{i}\mid i=1,\cdots,m\}$ be a simple-minded system of $\mathcal{F}(\mathcal{S})$ and $\mathcal{N}=\{N_{j}\mid i=1,\cdots,n\}$ a simple-minded system of $\mathcal{S}^{\bot}$. Then  $\mathcal{F}(\mathcal{M})=\mathcal{F}(\mathcal{S})$
and $\mathcal{F}(\mathcal{N})=\mathcal{S}^{\bot}$.
By Lemma \ref{omega-tri-subcat-0}, $\mathcal{S}^{\bot}={^{\bot}\mathcal{S}^{\bot}},$ then 
$i_{\ast}(\mathcal{M})$ and $i_{\ast}(\mathcal{N})$ are mutual orthogonal in $A$-$\stmod$. Thus $\mathcal{S}'$ is an orthogonal system in $A$-$\stmod$. Since $ (\mathcal{F}(\mathcal{S}),\mathcal{S}^{\bot})=(\mathcal{F}(\mathcal{M}),\mathcal{F}(\mathcal{N}))$ is a torsion pair  in $A$-$\stmod$, $A$-$\stmod=\mathcal{F}(\mathcal{M})\star\mathcal{F}(\mathcal{N})\subseteq\mathcal{F}(\mathcal{S}')$. Hence $A$-$\stmod=\mathcal{F}(\mathcal{S}')$ and $\mathcal{S}'$ is a simple-minded system  in $A$-$\stmod$. 
\end{proof}

For a subset  $\mathcal{S}_{1}$ of a simple-minded system $\mathcal{S}$ in  $A$-$\stmod$, by Theorem \ref{subset-of-sms}, there are two torsion pairs $(^{\bot}{\mathcal{S}_{1}},\mathcal{F}(\mathcal{S}_{1}))$ and  $(\mathcal{F}(\mathcal{S}_{1}),\mathcal{S}_{1}^{\bot})$ in $A$-$\stmod$. Thus we have the following corollary. 
\begin{Cor}\label{recollement-sms-1}
Let $A$ be a self-injective algebra and $\mathcal{S}$ a simple-minded  system in $A$-$\stmod$. Let  $\mathcal{S}_{1}$ be a Nakayama-stable subset of $\mathcal{S}$ such that  $\Omega(\mathcal{S}_{1})$ is contained in $\mathcal{F}(\mathcal{S}_{1})$.  Then  $A$-$\stmod$ is a recollement of $\mathcal{F}(\mathcal{S}_{1})$ and $\mathcal{S}_{1}^{\bot}$ as follows.
\begin{align}\label{recoll-2'}
\xymatrix{\mathcal{F}(\mathcal{S}_{1})\ar^-{i_*=i_!}[rr]&&A{\text-}\stmod\ar^-{j^!=j^*}[rr]\ar^-{i^!}@/^1.2pc/[ll]\ar_-{i^*}@/_1.6pc/[ll] &&\mathcal{S}_{1}^{\bot}\ar^-{j_*}@/^1.2pc/[ll]\ar_-{j_!}@/_1.6pc/[ll]}.
\end{align}
\end{Cor}

\begin{Rem}
\begin{enumerate}[$(1)$]
\item Note that $\ind(\mathcal{F}(\mathcal{S}_{1}))${\rm(}$\ind(\mathcal{S}_{1}^{\bot})$ or $\ind(^{\bot}\mathcal{S}_{1})${\rm)} is a union of a family of connected AR-components of the stable AR-quiver $_{s}\Gamma_{A}$. 
\item Under the conditions of Theorem \ref{recollement-sms} or  Corollary \ref{recollement-sms-1}, we have  $^{\perp}\mathcal{S}={^{\perp}\mathcal{S}^{\perp}}=\mathcal{S}^{\perp}$ or  $^{\perp}\mathcal{S}_{1}={^{\perp}\mathcal{S}_{1}^{\perp}}=\mathcal{S}_{1}^{\perp}$. Please see Lemma \ref{omega-tri-subcat-0}.  

\item There is no non-trivial  recollements in Corollary \ref{recollement-sms-1} for any  representation finite self-injective algebra. If $A$ is a representation-finite self-injective algebra, then there is no proper subset $\mathcal{S}_{1}$ of a simple-minded system satisfying the conditions of Corollary \ref{recollement-sms-1}. Please see Proposition \ref{sms-extending-RFS} for more details. 
\item There is no non-trivial t-structure for the stable module category of a symmetric algebra. Indeed, since the stable module category of a symmetric algebra is (-1)-Calabi-Yau, by Zhou and Zhu \cite[Proposition 4.6]{ZZ}, there is no non-trivial t-structure for the stable module category of a symmetric algebra. Our claim follows from the bijection of t-structures and recollements by Nicol\'{a}s and  Saor\'{\i}n \cite{NS}.
\item It would be interesting to provide an example of a non-trivial recollement satisfying the conditions in Theorem \ref{recollement-sms}.
\end{enumerate}
\end{Rem}

Let $A$ be a self-injective algebra.  Assume that $A$-$\stmod$  is a recollement of triangulated subcategories $\mathcal{X}$ and $\mathcal{Y}$
\begin{align}\label{recoll-3}
\xymatrix{\mathcal{X}\ar^-{i_*=i_!}[rr]&&A{\text-}\stmod\ar^-{j^!=j^*}[rr]
\ar^-{i^!}@/^1.2pc/[ll]\ar_-{i^*}@/_1.6pc/[ll]
&&\mathcal{Y}\ar^-{j_*}@/^1.2pc/[ll]\ar_-{j_!}@/_1.6pc/[ll]}.
\end{align}

By  Theorem \ref{TTF-triple-recollement}, there are two stable  t-structures $(\mathcal{U},\mathcal{V})$, $(\mathcal{V},\mathcal{W})$, where $\mathcal{U}=\I j_{!}$, $\mathcal{V}=\I i_{*}$ and $\mathcal{W}=\I j_{*}$. Note that $\mathcal{X}=\mathcal{V}$ and $\mathcal{Y}=\mathcal{W} \cong\mathcal{U}$. Since $\mathcal{X}$ and $\mathcal{Y}$ are triangulated subcategories, $(\mathcal{U},\mathcal{V}[-1])=(\mathcal{U},\mathcal{V})=(\mathcal{U},\mathcal{X})$ and $(\mathcal{V},\mathcal{W}[-1])=(\mathcal{V},\mathcal{W})=(\mathcal{X},\mathcal{Y})$ are torsion pairs in $A$-$\stmod$.

Let $\mathcal{M}$
(resp. $\mathcal{N}$) be a simple-minded system of $\mathcal{X}$ (resp. $\mathcal{Y}$). Then $\mathcal{F}(\mathcal{M})=\mathcal{X}$ and $\mathcal{F}(\mathcal{N})=\mathcal{Y}$. Thus  $(\mathcal{U},\mathcal{X})=(^{\bot}\mathcal{M},\mathcal{F}(\mathcal{M}))$ and $(\mathcal{X},\mathcal{Y})=(\mathcal{F}(\mathcal{M}),\mathcal{M}^{\bot})$.
If $\mathcal{M}$ is Nakayama-stable, that is, $\nu(\mathcal{M})=\mathcal{M}$, then by Lemma \ref{omega-tri-subcat-0},  $^{\perp}\mathcal{M}={^{\perp}\mathcal{M}^{\perp}}=\mathcal{M}^{\perp}$.
By Theorem \ref{recollement-sms} (2), we have the following corollary. 

\begin{Cor}\label{recollement-sms-2}
Let $A$ be a self-injective algebra and $A$-$\stmod$  a recollement of triangulated subcategories $\mathcal{X}$ and $\mathcal{Y}$ {\rm(}see diagram \ref{recoll-3}{\rm)}. Let $\mathcal{M}$
{\rm(}resp. $\mathcal{N}${\rm)} be a simple-minded system of $\mathcal{X}$ {\rm(}resp. $\mathcal{Y}${\rm)}.  If $\mathcal{M}$ is Nakayama-stable, then $\mathcal{S}'=i_{*}(\mathcal{M})\cup j_{*}(\mathcal{N})$  is a simple-minded system in $A$-$\stmod$.
\end{Cor}
We will not give  proof of the above corollary and we  remind readers  that the condition $\nu(\mathcal{M})=\mathcal{M}$ implies that each object $M_{i}$ of $\mathcal{M}^{\bot}$ is in $^{\bot}\mathcal{M}^{\bot}$ by Lemma \ref{subset-of-sms-bipen}. 


\subsection{Extendible property of simple-minded systems}
The authors proved in \cite[Section 3.2]{GLYZ} that, for a  self-injective algebra $A$ of finite type, every Nakayama-stable orthogonal system extends to a simple-minded system in $A$-$\stmod$. By adding one more condition, we obtain the following conclusion.

\begin{Prop}\label{sms-extending-RFS}
Let $A$ be a representation-finite self-injective algebra and $\mathcal{M}$ a Nakayama-stable orthogonal system in $A$-$\stmod$ such that $\Omega(\mathcal{M})\subseteq\mathcal{F}(\mathcal{M})$. Then $\mathcal{M}$ is a simple-minded system in $A$-$\stmod$. In particular, there is no non-trivial recollements in  Theorem \ref{recollement-sms} for any representation-finite self-injective algebra.
\end{Prop}
\begin{proof}
Let $A$ be a self-injective algebra of finite type. According to the classification of representation-finite self-injective algebras by Riedtmann (see for example \cite{Rie1, Rie3}), it is known that the stable AR-quiver $_{s}\Gamma_{A}$ of $A$ is of the form $\mathbb{Z}\Delta_{n}/<\sigma\tau^{\ell}>$, where $\Delta_{n}$ is a Dynkin quiver, $\mathbb{Z}\Delta$ is the stable translation quiver associated to $\Delta$, $\tau$ is the translation of $\mathbb{Z}\Delta$ and $\sigma$ is some automorphism of the quiver $\mathbb{Z}\Delta$ with a fixed vertex. 

Assume that $\mathcal{M}$ satisfies the given conditions.  Since $\mathcal{F}(\mathcal{M})$ is triangulated and $\mathcal{M}$ is Nakayama-stable, the $\tau$-orbit $\mathcal{W}=\{\tau^{i}(S)\mid i\in\mathbb{Z}\}$ of any $S\in\mathcal{M}$ is contained in $\mathcal{F}(\mathcal{M})$.  
Since  $\StHom_{A}(S,\nu\Omega(S))\cong\D\StHom_{A}(\nu\Omega(S),\nu\Omega(S))\cong\D\StHom_{A}(S,S)$, we have  $\dim_{k}{\StHom_{A}(S,\nu\Omega(S))}=\dim_{k}{\StHom_{A}(S,S)}=1$. Extend non-zero morphism $f\colon S\rightarrow\nu\Omega(S)$ in $A$-$\stmod$ to a triangle as follows.
\begin{align}\label{seq-0}\xymatrix{S\ar[r]^{f} &\nu\Omega(S)\ar[r]^{} & Z\ar[r]^{}& S[1].}
\end{align}
Rotating triangle (\ref{seq-0}) to the left once, we have the following  triangle \begin{align}\label{seq-2}\xymatrix{\nu\Omega(S)\ar[r]^{} & Z\ar[r]^{}& S[1]\ar[r]^{} &\nu\Omega(S)[1].}\end{align}
By \cite[Ch. V, Cor. 2.4]{ARS}, (\ref{seq-2}) is induced by the almost split sequence ended with $S[1]$. Thus any indecomposable, non-projective direct summands of $Z$ is not in the 
$\tau$-orbit $\mathcal{W}$ of $S$. Let $Z_{1}$ be an indecomposable, non-projective direct summand of $Z$ and take $\mathcal{N}=\{\tau^{i}(Z_{1})\mid i\in\mathbb{Z}\}$.  Hence  $\mathcal{F}(\mathcal{M})$  contains at least two $\tau$-orbits in $_{s}\Gamma_{A}$. Since  $\mathcal{F}(\mathcal{M})$ is closed under extension and $_{s}\Gamma_{A}$ is a connected component, it is not hard to know that $\mathcal{F}(\mathcal{M}\cup\mathcal{N})$ contains all  objects in  $_{s}\Gamma_{A}$. Thus  $\mathcal{F}(\mathcal{M})\supseteq\mathcal{F}(\mathcal{W}\cup\mathcal{N})=A$-$\stmod$. Hence $\mathcal{F}(\mathcal{M})=A$-$\stmod$ and $\mathcal{M}$ is a simple-minded system in $A$-$\stmod$.
\end{proof}

We presented  an example in \cite[Example 4.6]{Z0} that there is  a  Nakayama-stable orthogonal system which does not extend to a simple-minded system for a representation-infinite self-injective algebra. We shall study extendible property of simple-minded systems for arbitrary self-injective algebras. The following theorem is proved firstly by Coelho Sim\~{o}es and  Pauksztello \cite[Theorem  6.6]{CP}). We provide a new proof here. 
\begin{Them}\label{sms-extending}
Let $A$ be a self-injective algebra and $\mathcal{M}$ a  Nakayama-stable orthogonal system in $A$-$\stmod$ such that $\Omega(\mathcal{M})\subseteq\mathcal{F}(\mathcal{M})$ and $\mathcal{F}(\mathcal{M})$ is functorially finite in $A$-$\stmod$. Then $\mathcal{M}$ extends to  a simple-minded system in $A$-$\stmod$ if and only if there is a simple-minded system in $^{\bot}\mathcal{M}^{\bot}$.
\end{Them}
\begin{proof}
 If there is a simple-minded system $\mathcal{M}_{1}$ in $^{\bot}\mathcal{M}^{\bot}$, then $\mathcal{F}(\mathcal{M}_{1})={^{\bot}\mathcal{M}^{\bot}}$. By Theorem \ref{recollement-sms} (2), $\mathcal{M}$ extends to  a simple-minded system $\mathcal{S}$ in $A$-$\stmod$.

Conversely, if $\mathcal{M}$ extends to  a simple-minded system $\mathcal{S}$ in $A$-$\stmod$, then $\mathcal{S}\backslash\mathcal{M}$ is contained in  $^{\bot}\mathcal{M}^{\bot}$. We claim that $\mathcal{F}(\mathcal{S}\backslash\mathcal{M})={^{\bot}\mathcal{M}^{\bot}}$. It is clear that  $\mathcal{F}(\mathcal{S}\backslash\mathcal{M})\subseteq{^{\bot}\mathcal{M}^{\bot}}$. 
We prove ${^{\bot}\mathcal{M}^{\bot}}\subseteq\mathcal{F}(\mathcal{S}\backslash\mathcal{M})$. 
Take an object $X\in{^{\bot}\mathcal{M}^{\bot}}$.  Since $\mathcal{S}$  is a simple-minded system in $A$-$\stmod$, there is a positive integer $n$ such that $X\in(\mathcal{S})_{n}$. We induction on the positive integer $n$. It is easy  to know that the case for $n=1$ is true. We assume the case for $n=k$ holds, that is, if $X\in{^{\bot}\mathcal{M}^{\bot}}$ and $X\in(\mathcal{S})_{k}$, then $X\in\mathcal{F}(\mathcal{S}\backslash\mathcal{M})$. Take $n=k+1$. There is a triangle for $X$ as follows.
\begin{align}\label{tri-1}
\xymatrix{S\ar[r]^{ } &X\ar[r]^{} & W\ar[r]^{}& S[1],}
\end{align}
where $S\in\mathcal{S}$ and $W\in(\mathcal{S})_{k}$. Since $X\in{^{\bot}\mathcal{M}^{\bot}}$, $S\in\mathcal{S}\backslash\mathcal{M}$. 

By applying $\StHom_{A}(M_{i},-)$ for each $M_{i}\in\mathcal{M}$ to the triangle (\ref{tri-1}), there is an exact sequence (denote $\StHom_{A}(-,-)$ simply by $(-, -)$)
\begin{align}\label{tri-2}
\cdots\to (M_{i}, S)\to (M_{i},X)\to(M_{i},W)\to (M_{i},S[1])\to\cdots.
\end{align}
Since $S\in\mathcal{S}\backslash\mathcal{M}\in{^{\bot}\mathcal{M}^{\bot}}$ and $\mathcal{M}[-1]\subseteq\mathcal{F}(\mathcal{M})$, $\StHom_{A}(M_{i},S[1])\cong\StHom_{A}(M_{i}[-1],S)\cong0$. 
Since $\StHom_{A}(M_{i},X)=0=\StHom_{A}(M_{i},S[1])$, by exact sequence (\ref{tri-2}), $\StHom_{A}(M_{i},W)=0$.
By applying $\StHom_{A}(-,M_{i})$ for each $M_{i}\in\mathcal{M}$ to the triangle (\ref{tri-1}), we also have $\StHom_{A}(W,M_{i})=0$. Therefore $W\in{^{\bot}\mathcal{M}^{\bot}}.$ Since $W\in(\mathcal{S})_{k}$, by induction, we have $W\in\mathcal{F}(\mathcal{S}\backslash\mathcal{M})$. Hence  $X\in\mathcal{F}(\mathcal{S}\backslash\mathcal{M})$. Then  ${^{\bot}\mathcal{M}^{\bot}}\subseteq\mathcal{F}(\mathcal{S}\backslash\mathcal{M})$.
Thus $\mathcal{F}(\mathcal{S}\backslash\mathcal{M})={^{\bot}\mathcal{M}^{\bot}}$ and $\mathcal{S}\backslash\mathcal{M}$ is a simple-minded system of ${^{\bot}\mathcal{M}^{\bot}}$.
\end{proof}

\begin{Cor}\label{sms-one-one-corresp}
Let $A$ be a self-injective algebra and $\mathcal{M}$ a Nakayama-stable  orthogonal system such that $\Omega(\mathcal{M})\subseteq\mathcal{F}(\mathcal{M})$ and $\mathcal{F}(\mathcal{M})$ is functorially finite in $A$-$\stmod$. Then there is a one to one correspondence between the set of simple-minded systems containing $\mathcal{M}$  in  $A$-$\stmod$ and the set of  simple-minded systems in $^{\bot}\mathcal{M}^{\bot}$.
\end{Cor}

\begin{Cor}\label{sms-extending-2}
Let $A$ be a symmetric algebra and $\mathcal{M}$ an orthogonal system such that $\Omega(\mathcal{M})\subseteq\mathcal{F}(\mathcal{M})$ and $\mathcal{F}(\mathcal{M})$ is functorially finite in $A$-$\stmod$. Then there is a one to one correspondence between the set of simple-minded systems containing $\mathcal{M}$  in $A$-$\stmod$ and the set of  simple-minded systems in $^{\bot}\mathcal{M}^{\bot}$.
\end{Cor}

\begin{Rem}\label{bijective}
\begin{enumerate}[$(1)$]
\item On one hand, by Theorem  \ref{sms-extending}, if  $\mathcal{S}$   is a simple-minded system containing $\mathcal{M}$ in $A$-$\stmod$, then  $\mathcal{S}\backslash\mathcal{M}$ is a simple-minded system in $^{\bot}\mathcal{M}^{\bot}$. On the other hand, if $\mathcal{M}'$ is a simple-minded system in $^{\bot}\mathcal{M}^{\bot}$, then 
$\mathcal{M}'\cup\mathcal{M}$ is a simple-minded system containing $\mathcal{M}$ in $A$-$\stmod$.
\item This bijection is similar to Theorem \ref{sms-1-1-corresponding-1}. Note that $^{\bot}\mathcal{M}^{\bot}$ in Corollary \ref{sms-one-one-corresp} is a triangulated subcategory of  $A$-$\stmod$, however  the category $\mathcal{D}$ in Theorem \ref{sms-1-1-corresponding-1} is not a triangulated subcategory of $\mathcal{T}$. 

\item Note that there is a triangulated category which has no simple-minded system, please refer to \cite[Section 6]{Dugas}.
\item Proposition \ref{sms-extending-RFS} tell us that, for a self-injective algebra $A$ of finite type, there is no such  $\mathcal{M}$ in $A$-$\stmod$  satisfying conditions of Corollary \ref{sms-one-one-corresp}. It would be interesting to consider the case of  represenation-infinite self-injective algebras.
\end{enumerate}
\end{Rem}

\begin{Cor}\label{no-recollement-sms}
Let $A$ be a self-injective algebra, and let $\mathcal{S}$ be a Nakayama-stable orthogonal system in $A$-$\stmod$ such that  $\Omega(\mathcal{S})$ is contained in $\mathcal{F}(\mathcal{S})$ and  $\mathcal{F}(\mathcal{S})$ is functorially finite in $A$-$\stmod$.  If there is no non-trivial  recollements as in Theorem \ref{recollement-sms}, then $\mathcal{S}$ is a simple-minded system in $A$-$\stmod$.
\end{Cor}
 In Theorem \ref{BGA-sms}, We present  a sufficient and necessary condition for an  orthogonal system to be a simple-minded system over a Brauer graph algebra. We provide another characterization of simple-minded systems for domestic Brauer graph algebras in the following proposition. Please refer to \cite[Section 5.2]{S} for more detail about domestic Brauer graph algebras. 
\begin{Prop}\label{sms-domestic-BGA}
Let $A$ be a domestic Brauer graph algebra and $\mathcal{M}$ an orthogonal system such that $\Omega(\mathcal{M})\subseteq\mathcal{F}(\mathcal{M})$ and $\mathcal{F}(\mathcal{M})$ is functorially  finite  in $A$-$\stmod$. Then $\mathcal{M}$ is a simple-minded system in $A$-$\stmod$. In particular, there is no non-trivial recollements  in  Theorem \ref{recollement-sms} for any domestic Brauer graph algebra.
\end{Prop}
\begin{proof}
Note that  a Brauer graph algebra is a symmetric algebra, its  Nakayama functor is isomorphic to identity functor and stable Auslander-Reiten translation $\tau=[-2]$ in $A$-$\stmod$. Thus  any object in $A$-$\stmod$ is Nakayama-stable. It is clear that $\mathcal{M}$ is Nakayama-stable. 
Since $\mathcal{F}(\mathcal{M})$ is contravariantly finite  in $A$-$\stmod$, by Theorem \ref{two-torsion-pairs},    $(\mathcal{F}(\mathcal{M}),\mathcal{M}^{\bot})$ is a  torsion pair in $A$-$\stmod$. By Corollary \ref{omega-tri-subcat-1}, $\mathcal{M}^{\perp}$ is a triangulated subcategory. 

 If $\mathcal{F}(\mathcal{M})$ contains non-periodic modules, then the $\tau$-orbit $\{\tau^{i}M\mid i\in\mathbb{Z}\}$ of $M$ is contained in  $\mathcal{F}(\mathcal{M})$. Thus $\mathcal{M}$ is a simple-minded system in $A$-$\stmod$ by Theorem \ref{BGA-sms}. If  $\mathcal{F}(\mathcal{M})$ contains no non-periodic modules, then $\mathcal{M}^{\perp}$  contains non-periodic modules by Lemma \ref{splitting- torsion-pair}, since the torsion pair  $(\mathcal{F}(\mathcal{M}),\mathcal{M}^{\bot})$ is splitting. Thus $\mathcal{M}^{\perp}$  contains a $\tau$-orbit of a non-periodic module. By  the proof Theorem \ref{BGA-sms} (cf. \cite[Theorem 4.16]{Z}), $\mathcal{M}^{\perp}$  is the whole stable module category $A$-$\stmod$. It is a contradiction. Thus $\mathcal{M}$ is a simple-minded system in $A$-$\stmod$.
\end{proof}

\end{document}